\documentclass[12pt,twoside]{article}
\usepackage[latin1]{inputenc}
\usepackage[english]{babel}
\usepackage{clm-macros}
\usepackage{amsmath}
\usepackage{psfrag}
\usepackage{epsfig}
\usepackage{graphicx}
\usepackage{verbatim}
\usepackage{tikz}

\usepackage{datetime}

\usepackage{fullpage}

\usepackage{enumerate}

\usetikzlibrary{decorations}
\usetikzlibrary{decorations.pathmorphing}
\usetikzlibrary{arrows}
\tikzstyle{every node}=[circle, draw, fill=white,inner sep=0pt, minimum width=6pt]
\tikzstyle{nodelabel}=[rounded corners,fill=none,inner sep=5pt,draw=none]
\tikzstyle{matching} = [ultra thick]
\tikzset{snake/.style={decorate, decoration=snake}}

\newdateformat{UKvardate}{
\THEDAY\ \monthname[\THEMONTH], \THEYEAR}

\def\Notname{Notation}
\def\casname{Case}
\def\cnjname{Conjecture}
\def\corname{Corollary}
\def\Defname{Definition}
\def\exename{Exercise}
\def\exaname{Example}
\def\lemname{Lemma}
\def\listassertionname{List of Assertions}
\def\prbname{Problem}
\def\prpname{Proposition}
\def\proofOfname{\proofname{} of}
\def\ssbname{Case}
\def\subname{Case}
\def\thmname{Theorem}

\begin{document}

\newcommand{\BvN}{Birkhoff--von Neumann}
\newcommand{\PMc}{\mbox{PM-compact}}
\newcommand{\mcg}{matching covered graph}
\newcommand{\binv}{\mbox{$b$-invariant}}
\newcommand{\NT}{\mbox{Norine-Thomas}}

\title{Birkhoff--von Neumann graphs that are PM-compact}
\author{Marcelo H. de Carvalho$^{a}$\footnote{Supported by Fundect-MS and CNPq.} \and
            Nishad Kothari$^{b}$
            \footnote{Supported by Austrian Science Foundation FWF (START Y463)
            and by FAPESP Brazil (2018/04679-1).} \and
            Xiumei Wang$^{c}$\footnote{Supported by the National Natural Science Foundation of China (Nos. 11571323 and 11801526).} \and
            Yixun Lin$^{c}$\\
{\small $^{a}$ Institute of Computing, UFMS, Campo Grande, Brazil}\\
{\small $^{b}$ Faculty of Mathematics, University of Vienna, Austria}\\
{\small $^{c}$ School of Mathematics and Statistics,
Zhengzhou University, Zhengzhou, China}}

\UKvardate
\date{26 June, 2019}
  \maketitle 
  \thispagestyle{empty}

\begin{abstract}
A well-studied geometric object in combinatorial optimization is the
perfect matching polytope of a graph $G$ ---
the convex hull of the incidence vectors of all perfect matchings of $G$. In any investigation concerning the perfect
matching polytope, one may assume that $G$
is matching covered --- that is, $G$ is a connected graph (of order at least two) and each edge of $G$ lies in some perfect matching.

\smallskip
A graph $G$ is \BvN\ if its perfect matching polytope is characterized solely by non-negativity and
degree constraints. A result of Balas (1981) implies that $G$ is \BvN\ if and only if $G$ does not contain a pair of vertex-disjoint
odd cycles $(C_1,C_2)$ such that $G-V(C_1)-V(C_2)$ has a perfect matching. It follows immediately that the corresponding decision
problem is in co-$\mathcal{NP}$. However, it is not known to be in $\mathcal{NP}$. The problem is in $\mathcal{P}$
if the input graph is planar --- due to a result
of Carvalho, Lucchesi and Murty (2004). More recently, these authors, along with Kothari (2017), have shown that this problem
is equivalent to the seemingly unrelated problem of deciding whether a given graph is $\overline{C_6}$-free.

\smallskip
The combinatorial diameter of a polytope is the diameter of its $1$-skeleton graph.
A~graph $G$ is \PMc\ if the combinatorial diameter of its perfect matching polytope equals one.
Independent results of Balinski and Russakoff (1974), and of Chv{\'a}tal (1975), imply
that $G$ is \PMc\ if and only if $G$ does not contain a pair of vertex-disjoint
even cycles $(C_1,C_2)$ such that $G-V(C_1)-V(C_2)$ has a perfect matching.
Once again the corresponding decision problem is in co-$\mathcal{NP}$, but it is not known to be in $\mathcal{NP}$.
The problem is in $\mathcal{P}$
if the input graph is bipartite or is near-bipartite --- due to a result of Wang, Lin, Carvalho, Lucchesi, Sanjith and Little (2013).

\smallskip
In this paper, we consider the ``intersection'' of the aforementioned problems.
We give an alternative description of matching covered graphs
that are \BvN\ as well as \PMc; our description implies that the corresponding
decision problem is in $\mathcal{P}$.

\end{abstract}

\tableofcontents

\section{The perfect matching polytope}

Graphs considered in this paper are loopless; however, they may have multiple edges (joining any two vertices).
A graph is {\it simple} if it is devoid of multiple edges.
For a graph $G:=(V,E)$, we use $\mathbb{R}^E$ to denote the set of real vectors whose coordinates are indexed by the edges of $G$.
The {\it perfect matching polytope} of $G$, denoted by ${\mathcal Poly}(G)$, is the convex hull of the incidence vectors
of all perfect matchings of $G$.

\smallskip
For a set $S \subseteq V$, the {\it cut} of $S$, denoted by $\partial(S)$,
is the set of edges of $G$ that have one end in $S$, and the other end in $\overline{S}:=V(G)-S$.
We refer to $S$~and~$\overline{S}$ as the {\it shores} of the cut $\partial(S)$.
A cut is {\it trivial} if either shore is a singleton.
For a vertex $v$ of $G$, we simplify the notation $\partial(\{v\})$ to $\partial(v)$.
For a vector $x \in \mathbb{R}^{E}$ and a set $F \subseteq E$, we use $x(F)$ to denote $\sum_{e \in F} x(e)$.

\subsection{\BvN\ graphs}

For a graph $G:=(V,E)$, a vector $x \in \mathbb{R}^E$ is {\it $1$-regular} if $x(\partial(v)) = 1$
for each vertex $v$. Since the incidence vector corresponding to any perfect matching is non-negative and
$1$-regular, it follows that each vector in ${\mathcal Poly}(G)$ is non-negative and $1$-regular.
In the case of bipartite graphs, this obvious necessary condition is in fact sufficient, due to
the classical results of Birkhoff (1946) and of von Neumann (1953).

\begin{thm}
For a bipartite graph $G:=(V,E)$, a vector $x$ in $\mathbb{R}^E$ belongs to ${\mathcal Poly}(G)$ if and only if it is
non-negative and $1$-regular.
\end{thm}

A graph $G:=(V,E)$ is {\it \BvN} if it satisfies the property that a
vector $x$ in $\mathbb{R}^E$ belongs to ${\mathcal Poly}(G)$ if and only if it is
non-negative and $1$-regular. (Thus all bipartite graphs are \BvN.)
The complete graph $K_4$ is also \BvN, whereas the triangular prism $\overline{C_6}$
is not. This suggests the following problem.

\begin{prb}
\label{prb:BvN}
Characterize \BvN\ graphs. (Is the problem of deciding whether a given graph is \BvN\ in the complexity class $\mathcal{NP}$?
Is it in $\mathcal{P}$?)
\end{prb}

\subsection{\PMc\ graphs}

The {\it combinatorial diameter} of a polytope is the diameter of its $1$-skeleton graph.
A graph $G$ is {\it \PMc} if the combinatorial diameter of ${\mathcal Poly}(G)$ equals one, or equivalently,
if the $1$-skeleton graph of ${\mathcal Poly}(G)$ is a complete graph.
Balinski and Russakoff \cite{baru74},
and independently Chv{\'a}tal \cite{chva75},
showed that two vertices of the $1$-skeleton graph of ${\mathcal Poly}(G)$
are adjacent if and only if the symmetric difference of the corresponding perfect matchings
is exactly one (even) cycle. Consequently, a graph $G$ is \PMc\ if and only if the symmetric
difference of any two perfect matchings of $G$ is exactly one cycle. For instance, each of $K_4$ and $\overline{C_6}$
is \PMc, whereas the cube graph (on eight vertices) is not. This suggests the following problem.

\begin{prb}
\label{prb:PMc}
Characterize \PMc\ graphs. (Is the problem of deciding whether a given graph is \PMc\ in the complexity class $\mathcal{NP}$?
Is it in $\mathcal{P}$?)
\end{prb}

\subsection{Conformal bicycles}

A graph $G$ is {\it matchable} if it has a perfect matching.
A connected graph $G$, of order at least two, is {\it matching covered} if each edge belongs to some perfect
matching of $G$.

\smallskip
Using Edmond's algorithm, one may decide in polynomial-time
whether or not an edge belongs to some perfect matching.
Consequently,
it is not difficult to see that the search for an answer to Problem~\ref{prb:BvN},
or for an answer to Problem~\ref{prb:PMc}, may be restricted to {\mcg}s.

\smallskip
A subgraph $H$ of a matchable graph $G$ is {\it conformal} if $G-V(H)$
has a perfect matching. By a {\it conformal bicycle}, we mean a pair of vertex-disjoint cycles $(C_1,C_2)$ such that
\mbox{$G-V(C_1)-V(C_2)$} has a perfect matching. Since $|V(G)|$ is even, the parities of~$C_1$ and of~$C_2$ are the same.
We say that the conformal bicycle $(C_1,C_2)$ is {\it odd} if each of $C_1$ and $C_2$ is an odd cycle;
otherwise, it is {\it even}.

\smallskip
The aforementioned result of Balinski and Russakoff \cite{baru74}, and of Chv{\'a}tal \cite{chva75},
implies the following characterization
of \PMc\ matchable graphs; see \cite{wyl15} for a proof.

\begin{thm}
\label{thm:PMc-iff-no-even-conformal-bicycle}
A matchable graph is \PMc\ if and only if it does not contain an even conformal bicycle.
\end{thm}

Consequently, every simple matchable graph on at most six vertices
is \PMc. However, this is not necessarily true when we allow multiple edges. For example, $K_4$ is \PMc.
However, $K_4$ with multiple edges may have two vertex-disjoint cycles (each of length two), in which
case the graph is not \PMc. On the other hand, if a graph $G$ is \PMc\ then the underlying simple graph of $G$ is also \PMc.

\smallskip
If a matchable graph $G$ has an odd conformal bicycle, then it is easy to construct a non-negative
$1$-regular vector that does not belong to ${\mathcal Poly}(G)$, whence $G$ is not \BvN. A result of Balas \cite{bala81}
shows that the converse holds as well.

\begin{thm}
\label{thm:BvN-iff-no-odd-conformal-bicycle}
A matchable graph is \BvN\ if and only if it does not contain an odd conformal bicycle.
\end{thm}

Consequently, every matchable graph (not necessarily simple) on at most four vertices is \BvN.
Furthermore, a graph $G$ is \BvN\ if and only if the underlying simple graph of $G$ is \BvN.

\smallskip
It follows from
Theorems~\ref{thm:BvN-iff-no-odd-conformal-bicycle} and \ref{thm:PMc-iff-no-even-conformal-bicycle},
respectively, that
the problems of deciding whether a graph $G$ is \BvN, and of deciding whether $G$ is \PMc, both belong
to the complexity class co-$\mathcal{NP}$. However, to the best of our knowledge,
it is not known whether either of these problems is in $\mathcal{NP}$.

\subsection{Main result}

In this paper, we consider the ``intersection'' of these two problems.
We provide an exact characterization of {\mcg}s that are \BvN\ as well as \PMc.
Consequently, the problem of deciding whether a graph is \BvN\ as well as \PMc\
is in $\mathcal{P}$.
Before stating our result precisely, we briefly mention special cases of Problems~\ref{prb:BvN} and \ref{prb:PMc}
that have already been solved in the literature.

\smallskip
The problem of characterizing \BvN\ graphs is equivalent to another important problem in matching theory --- that of
characterizing ``solid'' graphs; see \cite{clm04}.
Carvalho, Lucchesi and Murty \cite{clm06} characterized the \BvN\ planar graphs.
Wang, Lin, Carvalho, Lucchesi, Sanjith and Little \cite{wlclsl13} characterized \PMc\ graphs
that are bipartite or ``near-bipartite'', whereas Wang, Shang, Lin and Carvalho \cite{wslc14}
characterized those that are cubic and claw-free.

\smallskip
Let $G$ be a \mcg\ and let $v$ be a vertex of degree two, with two distinct
neighbours $u$ and $w$. The {\it bicontraction} of $v$ is the operation of contracting the two edges $vu$ and $vw$
incident with $v$. The {\it retract} of $G$ is the graph obtained from $G$ by bicontracting all its degree two vertices.
It is easy to prove that the retract of a \mcg\ is also matching covered. Carvalho et al. \cite{clm05} showed that
the retract of a \mcg\ is unique up to isomorphism.
The following facts are easily proved
using the characterizations provided by
Theorems~\ref{thm:PMc-iff-no-even-conformal-bicycle}~and~\ref{thm:BvN-iff-no-odd-conformal-bicycle}, respectively.

\begin{prop}
\label{prop:PMc-retract}
A \mcg\ is \PMc\ if and only if its retract is \PMc.
\end{prop}

\begin{prop}
\label{prop:BvN-retract}
A \mcg\ is \BvN\ if and only if its retract is \BvN.
\end{prop}

Consequently, a \mcg\ $G$ is \BvN\ and \PMc\ if and only if the retract of $G$ has each of these properties.
Also, it is easily verified that, if $G$ is a \mcg\ of order at least four, and is not isomorphic to a simple cycle graph,
then the retract of $G$ has minimum degree three or more.
Thus, in order to characterize the {\mcg}s that are \BvN\ as well as \PMc, we may restrict attention to {\mcg}s
that have minimum degree three or more.

\smallskip
The following characterization of \PMc\ bipartite {\mcg}s was obtained by
Wang, Lin, Carvalho, Lucchesi, Sanjith and Little \cite{wlclsl13}.

\begin{thm}
\label{thm:bip-PMc-graphs}
The graphs $K_2$ (with multiple edges) and $K_{3,3}$ are the only \PMc\ \mbox{bipartite} {\mcg}s that have minimum degree
three or more.
\end{thm}

\smallskip
The {\it odd wheel} $W_{2k+1}$, for $k \geq 1$, is defined to be the join of an odd cycle $C_{2k+1}$ and $K_1$.
The smallest odd wheel is $K_4$.
If $k \geq 2$, then $W_{2k+1}$ has exactly one vertex of degree $2k+1$, called its {\it hub}, and the edges incident
at the hub are called its {\it spokes}. The remaining $2k+1$ vertices lie on a cycle, called the {\it rim}; vertices
and edges of the {\it rim} are referred to as {\it rim vertices} and {\it rim edges}, respectively.

\smallskip
Each cycle of $W_{2k+1}$ (possibly with multiple spokes), except for the rim, contains the hub;
consequently this graph does not have two vertex-disjoint cycles, whence it is \BvN\ as well as \PMc.

\begin{figure}[!htb]
\centering
\begin{tikzpicture}[scale=1]

\draw (0.75,1.5) -- (-1.5,0) -- (-0.75,1.5) -- (0,0) -- (0.75,1.5) -- (1.5,0) -- (-0.75,1.5);

\draw (0,0) -- (0,-1);
\draw (-1.5,0) -- (-1,-2);
\draw (1.5,0) -- (1,-2);

\draw (0,-1) -- (-1,-2) -- (1,-2) -- (0,-1);

\draw (0.75,1.5)node{}node[above,nodelabel]{$a_2$};
\draw (-0.75,1.5)node{}node[above,nodelabel]{$a_1$};

\draw (0,0)node{}node[left,nodelabel]{$b_2$};
\draw (-1.5,0)node{}node[left,nodelabel]{$b_1$};
\draw (1.5,0)node{}node[right,nodelabel]{$b_3$};

\draw (0,-1)node{}node[below,nodelabel]{$t_2$};
\draw (-1,-2)node{}node[left,nodelabel]{$t_1$};
\draw (1,-2)node{}node[right,nodelabel]{$t_3$};

\draw (0,-2.5)node[nodelabel]{(a)};
\end{tikzpicture}
\hspace*{1in}
\begin{tikzpicture}[scale=1]
\draw (-0.75,1.5) -- (0.75,1.5);

\draw (0.75,1.5) -- (-1.5,0) -- (-0.75,1.5) -- (0,0) -- (0.75,1.5) -- (1.5,0) -- (-0.75,1.5);

\draw (0,0) -- (0,-1);
\draw (-1.5,0) -- (-1,-2);
\draw (1.5,0) -- (1,-2);

\draw (0,-1) -- (-1,-2) -- (1,-2) -- (0,-1);

\draw (0.75,1.5)node{}node[above,nodelabel]{$a_2$};
\draw (-0.75,1.5)node{}node[above,nodelabel]{$a_1$};

\draw (0,0)node{}node[left,nodelabel]{$b_2$};
\draw (-1.5,0)node{}node[left,nodelabel]{$b_1$};
\draw (1.5,0)node{}node[right,nodelabel]{$b_3$};

\draw (0,-1)node{}node[below,nodelabel]{$t_2$};
\draw (-1,-2)node{}node[left,nodelabel]{$t_1$};
\draw (1,-2)node{}node[right,nodelabel]{$t_3$};

\draw (0,-2.5)node[nodelabel]{(b)};
\end{tikzpicture}
\vspace*{-0.2in}
\caption{(a) $K_4 \odot K_{3,3}$ ; (b) the Murty graph}
\label{fig:K4_splice_K33_and_Murty_graph}
\end{figure}
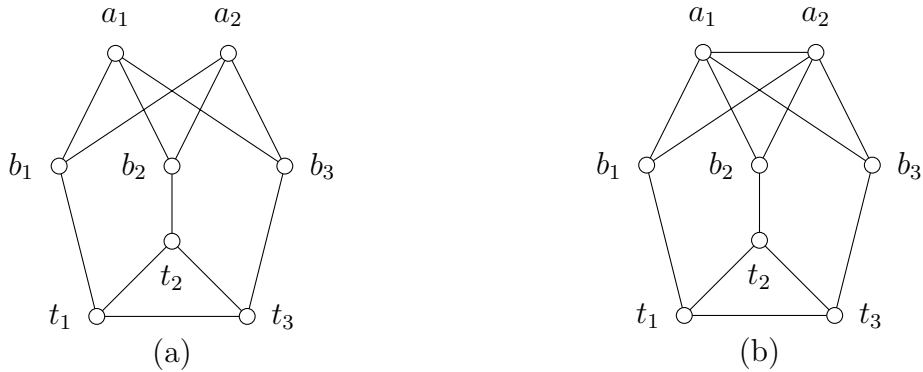

A vertex of a graph is {\it cubic} if its degree is precisely three;
otherwise it is {\it noncubic}.
A graph is {\it cubic} if all of its vertices are cubic.

\smallskip
We now describe two {\mcg}s that play an important role in our work.
We obtain these graphs from $K_{3,3}$ with color classes $\{a_1,a_2,a_3\}$ and $\{b_1,b_2,b_3\}$.
The first graph, denoted by $K_4 \odot K_{3,3}$, is cubic, and is obtained from $K_{3,3}$ by replacing the vertex~$a_3$
by a triangle, as shown in Figure~\ref{fig:K4_splice_K33_and_Murty_graph}(a). This graph is a ``splicing'' of $K_4$ and $K_{3,3}$.
The second graph, shown in Figure~\ref{fig:K4_splice_K33_and_Murty_graph}(b),
is obtained from $K_4 \odot K_{3,3}$ by adding an edge joining $a_1$~and~$a_2$. This graph has precisely two noncubic vertices,
and we shall refer to it as the Murty graph for reasons explained later.
It is worth noting that, for both graphs shown in Figure~\ref{fig:K4_splice_K33_and_Murty_graph},
the automorphism group has precisely three orbits --- $\{a_1,a_2\}, \{b_1,b_2,b_3\}$ and $\{t_1,t_2,t_3\}$.

\smallskip
The reader may check that the Murty graph (possibly with multiple
edges joining the two noncubic vertices) does not contain a conformal bicycle. Consequently, $K_4 \odot K_{3,3}$
does not contain a conformal bicycle either.
Our main result is the following.

\begin{thm}{\sc [The Main Theorem]}
\label{thm:main}
For a \mcg\ $G$, that has minimum degree three or more, the following are equivalent.
\begin{enumerate}[(i)]
\item $G$ is \BvN\ as well as \PMc.
\item $G$ does not contain a conformal bicycle.
\item $G$ is one of the following: (a) $K_2$ (with multiple edges), (b) $K_{3,3}$, (c) $K_4$ (up to multiple edges such that
it does not have two vertex-disjoint cycles), (d) an odd wheel of order six or more (up to multiple spokes), (e) $K_4 \odot K_{3,3}$,
(f) the Murty graph (up to multiple edges joining the two noncubic vertices).
\end{enumerate}
\end{thm}

\section{Matching covered graphs}

For a cut $C:=\partial(X)$ of a graph $G$, we denote the graph obtained by contracting the shore $X$ to a single vertex $x$ by
$G/(X \rightarrow x)$. In case the label of the contraction vertex $x$ is irrelevant, we simply write $G/X$.
The graphs $G/X$ and $G/\overline{X}$ are called the {\it $C$-contractions} of $G$, and
we say that $G$ is a {\it splicing} of these two graphs.

\subsection{Tight cut decomposition}

Let $G$ be a \mcg. A cut $C$ is a {\it tight cut} if $|M \cap C|=1$ for every perfect matching $M$ of $G$.
It is easily verified that if $C$ is a nontrivial tight cut of $G$, then each $C$-contraction is a \mcg\
that has strictly fewer vertices than $G$. If either of the $C$-contractions has a nontrivial tight cut, then that graph can be further
decomposed into even smaller {\mcg}s. We can repeat this procedure until we obtain a list of {\mcg}s, each of which
is free of nontrivial tight cuts. This procedure is known as a {\it tight cut decomposition} of $G$.

\smallskip
A matching covered graph free of nontrivial tight cuts is called a {\it brace} if it is bipartite; otherwise it is called a {\it brick}.
Thus a tight cut decomposition of a \mcg\ results in a list of bricks and braces. For example, a tight cut decomposition
of the graph $K_4 \odot K_{3,3}$, shown in Figure~\ref{fig:K4_splice_K33_and_Murty_graph}(a),
yields the brick $K_4$ and the brace $K_{3,3}$.
Lov{\'a}sz \cite{lova87} proved the following remarkable result.

\begin{thm}
Any two tight cut decompositions of a \mcg\ yield the same list of bricks and braces (except possibly for multiplicities
of edges).
\end{thm}

In particular, any two tight cut decompositions of a \mcg\ $G$ yield the same number of bricks; this number is denoted by $b(G)$.
We remark that $G$ is bipartite if and only if $b(G)=0$. We say that $G$ is a {\it near-brick} if $b(G)=1$.
(For instance, \mbox{$K_4 \odot K_{3,3}$}.)
Thus,
if $C$ is a nontrivial tight cut of a near-brick $G$, then one of its $C$-contractions is bipartite,
whereas the other $C$-contraction is a near-brick (on fewer vertices).

\smallskip
It is worth noting that each vertex of a brick, or of a brace on at least six vertices, has at least three distinct neighbours.

\subsection{Solid graphs}

The tight cut decomposition can be performed in polynomial-time, whence the bricks and braces
of a \mcg\ can be computed in polynomial-time.
Carvalho, Lucchesi and Murty \cite{clm04} proved the following ---
which implies that, in order to characterize \BvN\ graphs, it suffices to characterize \BvN\ bricks.

\begin{thm}
\label{thm:BvN-mcg-characterization}
A matching covered graph $G$ is \BvN\ if and only if either $G$ is bipartite or otherwise $G$ is a near-brick whose
unique brick is \BvN.
\end{thm}

We briefly discuss the class of ``solid'' {\mcg}s that plays a crucial role in several works of Carvalho, Lucchesi and Murty;
see \cite{clm02,clm04}. As per the terminology we have already defined, a matching covered graph $G$ is {\it solid}
if and only if each brick of $G$ (if any) is \BvN. Thus every \BvN\ \mcg\ is also solid. However, the converse is not true.
In the case of bricks (as well as near-bricks), the two notions coincide exactly.
Thus the problem of characterizing \BvN\ bricks is in fact the same as that of characterizing solid bricks.
The following result of Carvalho, Lucchesi and Murty \cite{clm06} implies that the problems of deciding whether a planar graph $G$
is \BvN, and of deciding whether $G$ is solid, are both in $\mathcal{P}$.

\begin{thm}
\label{thm:BvN-planar-bricks}
The odd wheels, up to multiple edges, are the only \BvN\ planar bricks.
\end{thm}

A graph $G$ is {\it odd-intercyclic} if it does not contain two vertex-disjoint odd cycles.
The odd wheels have this property.
Clearly, every odd-intercyclic brick is \BvN. The Murty graph, shown in Figure~\ref{fig:K4_splice_K33_and_Murty_graph}(b),
is the smallest brick that is \BvN, but is not odd-intercyclic.
It is in this context that U. S. R. Murty first stumbled upon this graph (private communication),
and it also appears in the work of Carvalho, Lucchesi and Murty \cite[Figure 14]{clm06}.

\subsection{Proof of Theorem~\ref{thm:main}}

As is often the case in matching theory, it turns out that the most difficult part of proving the Main Theorem (\ref{thm:main})
is when the graph under consideration is a brick. In this section, we will assume the following characterization of bricks
that are \BvN\ as well as \PMc, and present a proof of Theorem~\ref{thm:main} that relies on this assumption.

\begin{thm}
\label{thm:BvN-PMc-bricks}
A brick $G$ is \BvN\ as well as \PMc\ if and only if $G$ is one of the following: (a) $K_4$ (up to multiple edges such that
it does not have two vertex-disjoint cycles), (b) an odd wheel of order six or more (up to multiple spokes),
(c) the Murty graph (up to multiple edges joining the two noncubic vertices).
\end{thm}

A proof of the above result appears in Section~\ref{sec:bricks}. In order to prove Theorem~\ref{thm:main},
we will need a few more results.
The following is easy to prove using the characterization of \PMc\ graphs provided by
Theorem~\ref{thm:PMc-iff-no-even-conformal-bicycle}.
A proof appears in \cite{wlclsl13}.

\begin{prop}
\label{prop:PMc-inheritance-tight-cut-contractions}
Let $G$ be a \PMc\ \mcg, and let $C$ be a nontrivial tight cut of $G$. Then each $C$-contraction of $G$ is \PMc.
\end{prop}

Likewise, the following is easy to prove using the characterization of \BvN\ graphs provided by
Theorem~\ref{thm:BvN-iff-no-odd-conformal-bicycle}.

\begin{prop}
\label{prop:BvN-inheritance-tight-cut-contractions}
Let $G$ be a \BvN\ \mcg, and let $C$ be a nontrivial tight cut of $G$. Then each $C$-contraction of $G$ is \BvN.
\end{prop}

Finally, we will need the following lemma in order to prove Theorem~\ref{thm:main} using induction.

\begin{lem}
\label{lem:mcg-min-degree-three-or-more-tight-cut}
Let $G$ be a \mcg\ that has minimum degree three or more.
If $G$ has a nontrivial tight cut, then $G$ has a nontrivial tight cut that has at least three edges.
\end{lem}
\begin{proof}
Let $C:=\partial(X)$ be a nontrivial tight cut. If $|C| \geq 3$ then we are done.
Now suppose that $|C|=2$,
and let $C:=\{u\overline{u}, v\overline{v}\}$ such that $u, v \in X$ and $\overline{u},\overline{v} \in \overline{X}$.
We let $Y:=X - u + \overline{v}$ and $D:=\partial(Y)$.
Clearly, $|Y|=|X|$; consequently, $D$ is a nontrivial cut
and both of its shores are of odd cardinality.
Observe that each edge of $D$ is incident either with~$u$ or with~$\overline{v}$.
Thus each perfect matching meets $D$ in precisely one edge;
whence $D$ is a (nontrivial) tight cut.
Finally, we infer that $|D| \geq 4$ since $G$ has minimum degree three or more.
This completes the proof of Lemma~\ref{lem:mcg-min-degree-three-or-more-tight-cut}.
\end{proof}

We are now ready to prove the Main Theorem (\ref{thm:main}) --- under the assumption that Theorem~\ref{thm:BvN-PMc-bricks} holds.

\begin{proofOf}{Theorem~\ref{thm:main}}
The equivalence of the first two statements follows from
Theorems \ref{thm:PMc-iff-no-even-conformal-bicycle}~and~\ref{thm:BvN-iff-no-odd-conformal-bicycle}~.
As discussed earlier, the graphs mentioned in statement {\it (iii)} do not have a conformal bicycle,
whence {\it (iii)} implies {\it (ii)}. Our task is to prove that {\it (ii)} implies {\it (iii)}.

\smallskip
Suppose that $G$ is a matching covered graph that has minimum degree three or more and that does not have a conformal
bicycle. We will show that $G$ is one of the graphs mentioned in statement {\it (iii)}. We proceed by induction on the number
of edges.

\smallskip
If $G$ is bipartite, or if $G$ is a brick, then we are done by Theorem~\ref{thm:bip-PMc-graphs},
or by Theorem~\ref{thm:BvN-PMc-bricks}, respectively. Now suppose that $G$ is nonbipartite, and that it has a nontrivial tight cut.
Lemma~\ref{lem:mcg-min-degree-three-or-more-tight-cut} implies that $G$ has a nontrivial tight cut $C:=\partial(X)$
such that $|C| \geq 3$. We let $G_1:=G / \overline{X} \rightarrow \overline{x}$ and $G_2:=G / X \rightarrow x$
denote the two $C$-contractions of $G$. Each of them is a \mcg\ with fewer edges.
Since $|C| \geq 3$, each of $G_1$ and $G_2$ has minimum degree three or more.
Also, Propositions~\ref{prop:BvN-inheritance-tight-cut-contractions} and \ref{prop:PMc-inheritance-tight-cut-contractions} imply
that each of $G_1$ and $G_2$ is \BvN\ as well as \PMc.
Theorem~\ref{thm:BvN-mcg-characterization} implies that $G$ is a near-brick.
Consequently, exactly one of $G_1$ and $G_2$ is bipartite;
adjust notation so that $G_2$ is bipartite. It follows from the induction hypothesis that
$G_2$ is $K_{3,3}$. In particular, each vertex of $G_2$, including the contraction vertex~$x$, is cubic;
whence $|C|=3$ and the contraction vertex $\overline{x}$ of $G_1$ is also cubic.
We label the vertices of $G_2$ so that its color classes are $\{a_1,a_2,x\}$ and $\{b_1,b_2,b_3\}$.
By the induction hypothesis, we have four possibilities for the graph $G_1$, and we consider each of them
separately.

\medskip
\noindent
{\bf Case 1:} $G_1$ is $K_4$ (up to multiple edges).

\smallskip
\noindent
Note that the underlying simple graph of $G$ is $K_4 \odot K_{3,3}$. If $G$ is simple then there is nothing to prove.
Otherwise $G$ has a spanning subgraph $H$ that is isomorphic to the graph shown in Figure~\ref{fig:Cases-2.1-2.2} (left).
The reader may verify that $H$ has an even conformal bicycle, whence so does $G$, contrary to our assumption.

\medskip
\noindent
{\bf Case 2:} $G_1$ is an odd wheel of order six or more (up to multiple spokes).

\smallskip
\noindent
Since the contraction vertex $\overline{x}$ of $G_1$ is cubic, it lies on the rim of $G_1$.
We let $h$ denote the hub of $G_1$, and we label the remaining vertices of $G_1$ so that $G_1 - h$
is the cycle $(\overline{x}, w_1, w_2, \dots w_{2k},\overline{x})$.
Note that $k \geq 2$. Adjust notation so that the cut $C=\{b_1w_1,b_2h,b_3w_{2k}\}$.
Figure~\ref{fig:Cases-2.1-2.2} (right) shows the underlying simple graph when $k=2$.
Now let $C_1$ denote the even cycle $(w_1,w_2, \dots, w_{2k-1},h,w_1)$,
and let $C_2$ denote the $4$-cycle $(a_1,b_1,a_2,b_2,a_1)$. Observe that $(C_1,C_2)$
is an even conformal bicycle of $G$, contrary to our assumption.

\begin{figure}[!htb]
\centering
\begin{tikzpicture}[scale=1]
\draw (-1,-2) to [out=330,in=210] (1,-2);

\draw (0.75,1.5) -- (-1.5,0) -- (-0.75,1.5) -- (0,0) -- (0.75,1.5) -- (1.5,0) -- (-0.75,1.5);

\draw (0,0) -- (0,-1);
\draw (-1.5,0) -- (-1,-2);
\draw (1.5,0) -- (1,-2);

\draw (0,-1) -- (-1,-2) -- (1,-2) -- (0,-1);

\draw (0.75,1.5)node{};
\draw (-0.75,1.5)node{};

\draw (0,0)node{};
\draw (-1.5,0)node{};
\draw (1.5,0)node{};

\draw (0,-1)node{};
\draw (-1,-2)node{};
\draw (1,-2)node{};
\end{tikzpicture}
\hspace*{1in}
\begin{tikzpicture}[scale=1]
\draw (-1.5,-1.5) -- (-1,-3) -- (1,-3) -- (1.5,-1.5);

\draw (-1.5,-1.5) -- (0,-2);
\draw (-1,-3) -- (0,-2);
\draw (1,-3) -- (0,-2);
\draw (1.5,-1.5) -- (0,-2);

\draw (-1.5,0) -- (-1.5,-1.5);
\draw (0,0) -- (0,-2);
\draw (1.5,0) -- (1.5,-1.5);

\draw (0.75,1.5) -- (-1.5,0) -- (-0.75,1.5) -- (0,0) -- (0.75,1.5) -- (1.5,0) -- (-0.75,1.5);

\draw (0.75,1.5)node{}node[nodelabel,above]{$a_2$};
\draw (-0.75,1.5)node{}node[nodelabel,above]{$a_1$};

\draw (0,0)node{}node[nodelabel,left]{$b_2$};
\draw (-1.5,0)node{}node[nodelabel,left]{$b_1$};
\draw (1.5,0)node{}node[nodelabel,right]{$b_3$};

\draw (0,-2)node{}node[nodelabel,below]{$h$};
\draw (-1.5,-1.5)node{}node[nodelabel,left]{$w_1$};
\draw (-1,-3)node{}node[nodelabel,left]{$w_2$};
\draw (1,-3)node{}node[nodelabel,right]{$w_3$};
\draw (1.5,-1.5)node{}node[nodelabel,right]{$w_4$};

\end{tikzpicture}
\vspace*{-0.1in}
\caption{Illustrations for cases 1 and 2 in the proof of Theorem~\ref{thm:main}}
\label{fig:Cases-2.1-2.2}
\end{figure}
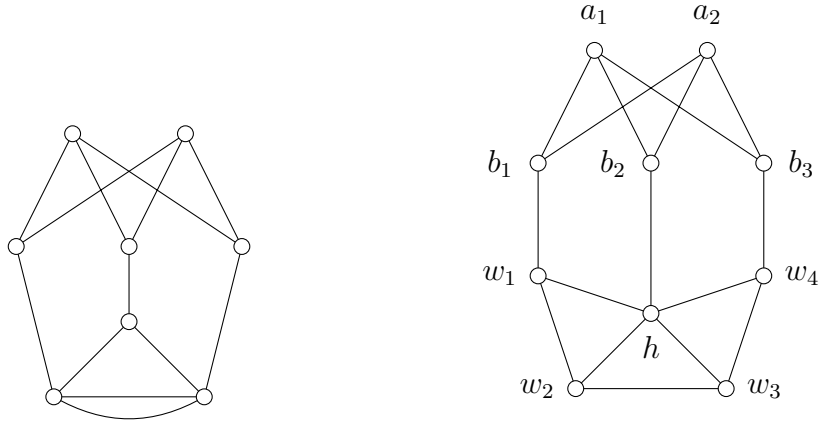

\medskip
\noindent
{\bf Case 3:} $G_1$ is $K_4 \odot K_{3,3}$.

\smallskip
\noindent
The graph $G$ is a splicing of two cubic graphs, namely $G_1 = K_4 \odot K_{3,3}$ and $G_2 = K_{3,3}$.
It follows from the automorphisms of $G_1$ that $G$ is isomorphic to one of three graphs shown in
Figure~\ref{fig:Cases-2.3-2.4}. Each of these three graphs has an even conformal bicycle that is shown using bold lines;
this contradicts our assumption.

\begin{figure}[!htb]
\centering
\begin{tikzpicture}[scale=1]
\draw[dashed] (-3,0.5) to [out=0,in=210] (0,1.5);
\draw (-3.3,0.5)node[nodelabel]{$C$};

\draw (-3,1.5) -- (-1.5,0);
\draw (-2,1.5) -- (0,0);
\draw (-1,1.5) -- (1.5,0);

\draw (-1.5,2.5) -- (-3,1.5) -- (-2.5,2.5) -- (-2,1.5) -- (-1.5,2.5) -- (-1,1.5) -- (-2.5,2.5);

\draw (0.75,1.5) -- (-1.5,0);
\draw (0.75,1.5) -- (0,0);
\draw (0.75,1.5) -- (1.5,0);

\draw (0,0) -- (0,-1);
\draw (-1.5,0) -- (-1,-2);
\draw (1.5,0) -- (1,-2);

\draw (0,-1) -- (-1,-2) -- (1,-2) -- (0,-1);

\draw[ultra thick] (0.75,1.5) -- (1.5,0) -- (1,-2) -- (0,-1) -- (-1,-2) -- (-1.5,0) -- (0.75,1.5);
\draw[ultra thick] (-1.5,2.5) -- (-3,1.5) -- (-2.5,2.5) -- (-1,1.5) -- (-1.5,2.5);

\draw (0.75,1.5)node{};

\draw (0,0)node{};
\draw (-1.5,0)node{};
\draw (1.5,0)node{};

\draw (0,-1)node{};
\draw (-1,-2)node{};
\draw (1,-2)node{};

\draw (-2,1.5)node{};
\draw (-3,1.5)node{};
\draw (-1,1.5)node{};

\draw (-1.5,2.5)node{};
\draw (-2.5,2.5)node{};

\draw (0,-3)node[nodelabel]{$(a)$};

\end{tikzpicture}
\hspace*{0.5in}
\begin{tikzpicture}[scale=1]
\draw[dashed] (0,0) circle (1.15);
\draw (0.9,0)node[nodelabel]{$C$};

\draw (0,-0.75) -- (0,-1.5);
\draw (-0.5,0.75) -- (-0.75,2.5);
\draw (0.5,0.75) -- (0.75,2.5);

\draw (-0.5,0) -- (-0.5,0.75);
\draw (-0.5,0) -- (0.5,0.75);
\draw (-0.5,0) -- (0,-0.75);
\draw (0.5,0) -- (-0.5,0.75);
\draw (0.5,0) -- (0.5,0.75);
\draw (0.5,0) -- (0,-0.75);

\draw (-0.75,2.5) -- (-2,0);
\draw (-0.75,2.5) -- (2,0);
\draw (0.75,2.5) -- (-2,0);
\draw (0.75,2.5) -- (2,0);

\draw (-2,0) -- (-1,-2.5);
\draw (2,0) -- (1,-2.5);

\draw (0,-1.5) -- (-1,-2.5) -- (1,-2.5) -- (0,-1.5);

\draw[ultra thick] (-0.75,2.5) -- (-2,0) -- (0.75,2.5) -- (2,0) -- (-0.75,2.5);
\draw[ultra thick] (-0.5,0) -- (-0.5,0.75) -- (0.5,0) -- (0.5,0.75) -- (-0.5,0);

\draw (0.75,2.5)node{}node[above,nodelabel]{$v$};
\draw (-0.75,2.5)node{}node[above,nodelabel]{$u$};

\draw (-2,0)node{};
\draw (2,0)node{};

\draw (0,-1.5)node{};
\draw (-1,-2.5)node{};
\draw (1,-2.5)node{};

\draw (0,-0.75)node{};
\draw (-0.5,0.75)node{};
\draw (0.5,0.75)node{};

\draw (-0.5,0)node{};
\draw (0.5,0)node{};

\draw (0,-3.5)node[nodelabel]{$(b)$};

\end{tikzpicture}
\hspace*{0.5in}
\begin{tikzpicture}[scale=1]
\draw[dashed] (0,-1.5) circle (1.15);
\draw (0.9,-1.5)node[nodelabel]{$C$};

\draw (0,-0.75) -- (0,0);
\draw (-0.5,-2.25) -- (-1.5,-3.5);
\draw (0.5,-2.25) -- (1.5,-3.5);

\draw (-0.5,-1.5) -- (0,-0.75);
\draw (-0.5,-1.5) -- (-0.5,-2.25);
\draw (-0.5,-1.5) -- (0.5,-2.25);
\draw (0.5,-1.5) -- (0,-0.75);
\draw (0.5,-1.5) -- (-0.5,-2.25);
\draw (0.5,-1.5) -- (0.5,-2.25);

\draw (0.75,1.5) -- (-1.5,0) -- (-0.75,1.5) -- (0,0) -- (0.75,1.5) -- (1.5,0) -- (-0.75,1.5);

\draw (-1.5,0) -- (-1.5,-3.5);
\draw (1.5,0) -- (1.5,-3.5);

\draw (-1.5,-3.5) -- (1.5,-3.5);

\draw[ultra thick] (-1.5,0) -- (-0.75,1.5) -- (1.5,0) -- (0.75,1.5) -- (-1.5,0);
\draw[ultra thick] (-0.5,-1.5) -- (-0.5,-2.25) -- (0.5,-1.5) -- (0.5,-2.25) -- (-0.5,-1.5);

\draw (0.75,1.5)node{}node[above,nodelabel]{$v$};
\draw (-0.75,1.5)node{}node[above,nodelabel]{$u$};

\draw (0,0)node{};
\draw (-1.5,0)node{};
\draw (1.5,0)node{};

\draw (-1.5,-3.5)node{};
\draw (1.5,-3.5)node{};

\draw (0,-0.75)node{};
\draw (-0.5,-2.25)node{};
\draw (0.5,-2.25)node{};

\draw (-0.5,-1.5)node{};
\draw (0.5,-1.5)node{};

\draw (0,-4.4)node[nodelabel]{$(c)$};

\end{tikzpicture}
\vspace*{-0.2in}
\caption{Illustrations for cases 3 and 4 in the proof of Theorem~\ref{thm:main}}
\label{fig:Cases-2.3-2.4}
\end{figure}

\medskip
\noindent
{\bf Case 4:} $G_1$ is the Murty graph (up to multiple edges joining the two noncubic vertices).

\smallskip
\noindent
The graph $G$ is a splicing of the Murty graph (that is, $G_1$) and $G_2 = K_{3,3}$.
Let $u$ and $v$ denote the two noncubic vertices of $G_1$.
Since the contraction vertex $\overline{x}$
of $G_1$ is cubic, it is distinct from $u$ and $v$.
It follows from the automorphisms of $G_1$ that the underlying simple graph of $G$
is isomorphic to either the graph shown in Figure~\ref{fig:Cases-2.3-2.4}$(b)$ plus the edge $uv$,
or to the graph shown in Figure~\ref{fig:Cases-2.3-2.4}$(c)$ plus the edge $uv$.
In either case, $G$ contains an even conformal bicycle, contrary to our assumption.

\bigskip
This completes the proof of the Main Theorem~(\ref{thm:main}).
\end{proofOf}

\section{Bricks}
\label{sec:bricks}

Recall that bricks are those nonbipartite matching covered graphs that are free of nontrivial tight cuts.
Edmonds, Lov{\'a}sz and Pulleyblank \cite{elp82} proved the following deep result.

\begin{thm}
A graph $G$, of order at least four, is a brick if and only if $G-u-v$ is connected and matchable, for all pairs of distinct vertices
$u,v \in V(G)$.
\end{thm}

A consequence worth noting is that adding an edge to a brick yields another brick.
In this section, our goal is to prove Theorem~\ref{thm:BvN-PMc-bricks} that provides a complete characterization of bricks
that are \BvN\ and \PMc. However, before that, we discuss induction tools from existing literature
that are useful in proving results concerning bricks.

\subsection{Strictly thin edges}

An edge $e$ of a brick $G$ is {\it thin} if the retract of $G-e$ is also a brick.
Carvalho, Lucchesi and Murty \cite{clm06} proved the following result.

\begin{thm}
\label{thm:thin-edge}
Every brick, distinct from $K_4$, $\overline{C_6}$ and the Petersen graph, has a thin edge.
\end{thm}

Carvalho et al. \cite{clm06} also described four simple expansion operations that may be applied to any brick to obtain
a larger brick. Given a brick $H$, the application of any of those four operations to $H$ results in a brick $G$
such that $G$ has a thin edge $e$ with the property that $H$ is the retract of $G-e$. We refer the reader to their
work for a precise description of the expansion operations. Their importance is due to the following result; see \cite[Theorem 36]{clm06}.

\begin{thm}
\label{thm:brick-expansion}
Let $H$ be a brick and let $G$ be a graph obtained from $H$ by one of the four expansion operations. Then $G$ is a brick.
\end{thm}

Thus, any brick may be
generated from one of the three bricks ($K_4, \overline{C_6}$ and the Petersen graph) by means of the
four expansion operations.
In the same paper, Carvalho et al. used this generation procedure to prove
Theorem~\ref{thm:BvN-planar-bricks} that provides a complete characterization
of \BvN\ planar bricks.

\smallskip
However, one of the problems with this generation procedure is that in order to generate certain simple bricks\footnote{A simple brick is a brick that is devoid of multiple edges.}
(such as the odd wheels of order six or more) one may have to allow intermediate bricks that are not
necessarily simple. When proving other results using this generation procedure,
the presence of multiple edges requires greater case analyis. In order to circumvent this difficulty, we shall instead
use the stronger notion of a strictly thin edge, as is done in the work of Kothari and Murty \cite{komu16}.

\smallskip
A thin edge $e$ of a simple brick $G$ is {\it strictly thin} if the retract of $G-e$ is a simple brick. With each strictly thin
edge $e$ of $G$ there is an associated number, called its {\it index}, which is:

\begin{itemize}
\item {\it zero}, if both ends of $e$ are noncubic in $G$;
\item {\it one}, if exactly one end of $e$ is cubic in $G$;
\item {\it two}, if both ends of $e$ are cubic in $G$, and $e$ does not lie in a triangle;
\item {\it three}, if both ends of $e$ are cubic in $G$, and $e$ lies in a triangle.
\end{itemize}

The following proposition is easily verified (and appears in \cite{komu16}).

\begin{prop}
\label{prop:brick-minus-thin-edge-cases}
Let $G$ be a simple brick, and let $e$ denote a strictly thin edge of $G$, and let $H$ be the retract of $G-e$.
If the index of $e$ is zero, then $H=G-e$.
If the index of $e$ is one, then $G-e$ has precisely one vertex of degree two;
and $H$ has just one contraction vertex, and its degree is at least four.
If the index of $e$ is two, then $G-e$ has precisely two vertices of degree two, and they have no common neighbour;
and $H$ has two contraction vertices, and their degrees are at least four.
If the index of $e$ is three, then $G-e$ has precisely two vertices of degree two, and they have a common neighbour;
and $H$ has just one contraction vertex, and its degree is at least five.
\end{prop}

We refer the reader to \cite[Figure 2]{komu16} for examples of strictly thin edges. It is easily verified that
odd wheels do not have strictly thin edges. There are four other infinite families of simple bricks that are devoid of strictly thin edges.
These are prisms, truncated biwheels, staircases and M{\"o}bius ladders. We refer the reader to \cite{komu16} for a description of these
families. Norine and Thomas \cite{noth07} proved the following.

\begin{thm}
\label{thm:strictly-thin-edge}
Let $G$ be a simple brick. If $G$ has no strictly thin edge then $G$ is either the Petersen graph or is an odd wheel, a prism, a truncated biwheel,
a staircase or a M{\"o}bius ladder.
\end{thm}

As is done in \cite{komu16}, we say that a simple brick is {\it \NT} if it is free of strictly thin edges.
(Thus, these are exactly the graphs that appear in the above theorem statement.)
The following is an immediate consequence of Theorem~\ref{thm:strictly-thin-edge}.

\begin{thm}
\label{thm:strictly-thin-generation}
Given any simple brick $G$, there exists a sequence $G_1, G_2, \dots, G_r$ of simple bricks such that (i) $G_1$
is a \NT\ brick; (ii) $G_r=G$; and (iii) for $2 \leq i \leq r$, $G_i$ has a strictly thin edge $e_i$ such that $G_{i-1}$
is the retract of $G_i-e_i$.
\end{thm}

The above theorem implies that every simple brick can be generated from one of the \NT\ bricks by repeated
application of the expansion operations in some sequence, such that at each step we have a simple brick.
Kothari and Murty \cite{komu16} used this generation procedure to provide complete characterizations of planar bricks
that are $K_4$-free, and those that are $\overline{C_6}$-free.

\smallskip
We will use this generation procedure in order to prove Theorem~\ref{thm:BvN-PMc-bricks}.
Let $G$ be a brick and let $e$ be a thin edge of $G$. If $G$ is \PMc, then it follows from
Theorem~\ref{thm:PMc-iff-no-even-conformal-bicycle}
that $G-e$ is also \PMc, and so is the retract of $G-e$ by invoking Proposition~\ref{prop:PMc-retract}.
A similar fact holds for the \BvN\ property by invoking Theorem~\ref{thm:BvN-iff-no-odd-conformal-bicycle}
and Proposition~\ref{prop:BvN-retract}.

\begin{prop}
\label{prop:BvN-PMc-retract}
Let $G$ be a brick, let $e$ be a thin edge of $G$, and let $H$ be the retract of $G-e$. Then the following statements hold.
\begin{enumerate}[(i)]
\item If $G$ is \BvN\ then $H$ is also \BvN.
\item If $G$ is \PMc\ then $H$ is also \PMc.
\end{enumerate}
\end{prop}

In order to characterize the bricks that are \BvN\ as well as \PMc, we first need to find out which \NT\ bricks have
both of these properties.
This is done easily using Theorems~\ref{thm:PMc-iff-no-even-conformal-bicycle}~and~\ref{thm:BvN-iff-no-odd-conformal-bicycle}.

\begin{prop}
\label{prop:NT-brick-BvN-PMc}
The odd wheels are the only \NT\ bricks that are \BvN\ as well as \PMc.
\end{prop}
\begin{proof}
Let $G$ be any \NT\ brick.
If $G$ is a prism, or a truncated biwheel, or a staircase, or the Petersen graph,
then $G$ has an odd conformal bicycle, whence $G$ is not \BvN. On the other hand, if $G$ is a M{\"o}bius ladder of order
eight or more, then $G$ has an even conformal bicycle. (The only other M{\"o}bius ladder is $K_4$ which is also an odd wheel.)
Consequently, if $G$ is \BvN\ and \PMc, then $G$ is indeed an odd wheel.
\end{proof}

It follows from Propositions~\ref{prop:BvN-PMc-retract}~and~\ref{prop:NT-brick-BvN-PMc} that,
if $G$ is a simple brick that is \BvN\ and \PMc, and if $G_1, G_2, \dots, G_r$ is a sequence of simple bricks as in
Theorem~\ref{thm:strictly-thin-generation}, then each $G_i$ is also \BvN\ and \PMc, and furthermore, $G_1$ is in fact an odd wheel.

\smallskip
In order to prove Theorem~\ref{thm:BvN-PMc-bricks}, we need to investigate how ``new'' bricks that are \BvN\ as well as \PMc\ may
be generated from the odd wheels, and also from the Murty graph (which is in turn generated from $W_5$ as we will see soon).
The next four sections are devoted to this task. Most of the proofs are straightforward.
We will make extensive use of Theorems~\ref{thm:PMc-iff-no-even-conformal-bicycle}~and~\ref{thm:BvN-iff-no-odd-conformal-bicycle}
without referring to them explicitly.

\subsection{Adding an edge}

In this section, we consider the simplest operation --- that of adding an edge.

\begin{prop}
\label{prop:wheel_plus_edge}
Let $G$ be a graph obtained from $W_{2k+1}$, where $k \geq 2$, by adding an edge~$e$ joining any two rim vertices.
Then either $G$ is not \BvN, or $G$ is not \PMc, possibly both.
\end{prop}
\begin{proof}
We let $u$ and $v$ denote the ends of $e$, and let $h$ denote the hub of $W_{2k+1}$.
If $u$~and~$v$ are adjacent in $W_{2k+1}$, then there is a $2$-cycle $C_1$ containing $u$ and $v$,
and there is a cycle containing all of the remaining $2k$ vertices, whence $(C_1,C_2)$ is an even conformal bicycle and $G$ is not \PMc.
Now suppose that $u$ and $v$ are nonadjacent in $W_{2k+1}$. The rim of $W_{2k+1}$ is a union of two internally-disjoint $uv$-paths,
say $P_1$ and $P_2$. One of them, say $P_1$, is of odd length, and has length at least three.
We let $Q_1 := P_1 - u - v$, and let $u'$ and $v'$ denote the ends of $Q_1$. Let $C_1 := Q_1+v'h+hu'$ and $C_2:=P_2+e$.
Observe that
$(C_1,C_2)$ is an odd conformal bicycle, whence $G$ is not \BvN.
\end{proof}

\begin{prop}
\label{prop:Murty_graph_plus_edge}
Let $H$ denote the Murty graph, and let $a_1$ and $a_2$ denote its noncubic vertices.
Let $G$ be a graph obtained from $H$ by adding an edge $e$ such that at least one of $a_1$ and $a_2$ is not an end of $e$.
Then either $G$ is not \BvN, or $G$ is not \PMc, possibly both.
\end{prop}
\begin{proof}
We label the graph $H$ as in Figure~\ref{fig:K4_splice_K33_and_Murty_graph}(b), and we let $u$ and $v$ denote the ends of~$e$.
First suppose that $u$ and $v$ are adjacent in $H$. Up to symmetry, there are three cases: $\{u,v\} = \{t_1,t_3\}$,
$\{u,v\} = \{t_1,b_1\}$ and $\{u,v\} = \{a_1,b_1\}$. In each case, $G$ has an even conformal bicycle.
Now suppose that $u$ and $v$ are nonadjacent in $H$. Up to symmetry, there are three cases:
$\{u,v\} = \{b_1,b_2\}$, $\{u,v\} = \{t_1,b_2\}$ and $\{u,v\} = \{t_1,a_1\}$.
In each case, $G$ has either an even conformal bicycle, or an odd conformal bicycle.
(We omit the details.)
\end{proof}

\bigskip
In the next three sections, we consider the remaining three expansion operations. However, for convenience, we shall instead
think of deleting a strictly thin edge $e$ from a brick $G$ and taking the retract in order to obtain a smaller brick $H$
that is either an odd wheel or is the Murty graph.
We will invoke Proposition~\ref{prop:brick-minus-thin-edge-cases} without referring to it explicitly.
We shall adopt the following notation and conventions.

\begin{Not}
\label{Not:strictly_thin_edge}
For a strictly thin edge $e$ of a simple brick $G$,
we let $e:=u_0v_0$, and we let $H$ denote the retract of $G-e$.
If $u_0$ is cubic in $G$, we let $u_1$ and $u_2$ denote the neighbours of $u_0$ that are distinct from $v_0$.
Likewise, if $v_0$ is cubic in $G$, we let $v_1$ and $v_2$ denote the neighbours of $v_0$ that are distinct from $u_0$.
If the index of $e$ is one then we adjust notation so that $u_0$ is cubic and $v_0$ is noncubic.
If the index of $e$ is three then $u_0$ and $v_0$ have a common neighbour, and we adjust notation so that $u_1=v_1$.
\end{Not}

\subsection{Index one}

\begin{prop}
\label{prop:Murty_graph_index_1}
Let $G$ be a simple brick, and let $e$ be a strictly thin edge of index one such that the retract of $G-e$ is the Murty graph.
Then either $G$ is not \BvN, or $G$ is not \PMc, possibly both.
\end{prop}
\begin{proof}
We adopt Notation~\ref{Not:strictly_thin_edge}.
We label the vertices of $H$
as shown in Figure~\ref{fig:K4_splice_K33_and_Murty_graph}(b). The contraction
vertex of $H$ resulting from the bicontraction of $u_0$ has degree at least four, whence one of $a_1$ and $a_2$ is the contraction vertex.
Adjust notation so that $a_1$ is the contraction vertex.
The graph $H-a_1$ is shown in Figure~\ref{fig:Murty_graph_index_1}.
In~$G$, precisely two vertices from the set $\{a_2,b_1,b_2,b_3\}$ are neighbours of~$u_1$,
and the remaining two are neighbours of~$u_2$.
By observing the symmetries of $H-a_1$, we may adjust notation so that $b_1$~and~$b_2$ are neighbours of $u_1$, whereas
$b_3$ and $a_2$ are neighbours of $u_2$, as shown in Figure~\ref{fig:Murty_graph_index_1}.

\begin{figure}[!htb]
\centering
\begin{tikzpicture}[scale=1]

\draw (0,1.5) -- (-1.5,0);
\draw (0,1.5) -- (0,0);
\draw (0,1.5) -- (1.5,0);

\draw (0,0) -- (0,-1);
\draw (-1.5,0) -- (-1,-2);
\draw (1.5,0) -- (1,-2);

\draw (0,-1) -- (-1,-2) -- (1,-2) -- (0,-1);

\draw (0,1.5)node{}node[above,nodelabel]{$a_2$};

\draw (0,0)node{}node[left,nodelabel]{$b_2$};
\draw (-1.5,0)node{}node[left,nodelabel]{$b_1$};
\draw (1.5,0)node{}node[right,nodelabel]{$b_3$};

\draw (0,-1)node{}node[below,nodelabel]{$t_2$};
\draw (-1,-2)node{}node[left,nodelabel]{$t_1$};
\draw (1,-2)node{}node[right,nodelabel]{$t_3$};

\draw (0,-2.5)node[nodelabel]{$H-a_1$};
\end{tikzpicture}
\hspace*{0.5in}
\begin{tikzpicture}[scale=1]
\draw (-2.3,1.5) -- (-1.7,1.9) -- (-1.1,2.3);
\draw (-2.3,1.5) -- (-1.5,0);
\draw (-2.3,1.5) -- (0,0);
\draw (-1.1,2.3) -- (1.5,0);
\draw (-1.1,2.3) -- (0.75,1.5);


\draw (0.75,1.5) -- (-1.5,0);
\draw (0.75,1.5) -- (0,0);
\draw (0.75,1.5) -- (1.5,0);

\draw (0,0) -- (0,-1);
\draw (-1.5,0) -- (-1,-2);
\draw (1.5,0) -- (1,-2);

\draw (0,-1) -- (-1,-2) -- (1,-2) -- (0,-1);

\draw (0.75,1.5)node{}node[above,nodelabel]{$a_2$};

\draw (0,0)node{}node[left,nodelabel]{$b_2$};
\draw (-1.5,0)node{}node[left,nodelabel]{$b_1$};
\draw (1.5,0)node{}node[right,nodelabel]{$b_3$};

\draw (0,-1)node{}node[below,nodelabel]{$t_2$};
\draw (-1,-2)node{}node[left,nodelabel]{$t_1$};
\draw (1,-2)node{}node[right,nodelabel]{$t_3$};

\draw (-2.3,1.5)node{}node[above left, nodelabel]{$u_1$};
\draw (-1.7,1.9)node{}node[above left, nodelabel]{$u_0$};
\draw (-1.1,2.3)node{}node[above, nodelabel]{$u_2$};

\draw (0,-2.5)node[nodelabel]{$G-e$};
\end{tikzpicture}
\vspace*{-0.3in}
\caption{Illustration for the proof of Proposition~\ref{prop:Murty_graph_index_1}}
\label{fig:Murty_graph_index_1}
\end{figure}

In $G-e$, there is an automorphism that swaps $t_1$ and $t_2$, as well as $b_1$ and $b_2$, and keeps remaining vertices fixed.
Thus, up to symmetry, there are five possibilities for the end $v_0$ of the edge $e$. These are: $v_0 \in \{t_1, b_1, t_3, b_3, a_2\}$.
In each case, we present a conformal bicycle $(C_1,C_2)$.
If $v_0 = t_1$ then let $C_1:=(t_1,u_0,u_1,b_1,t_1)$ and $C_2:=(b_3,u_2,a_2,b_2,t_2,t_3,b_3)$.
If $v_0 = b_1$ then let $C_1:=(b_1,u_0,u_1,b_1)$ and $C_2:=(u_2,a_2,b_3,u_2)$.
If $v_0 = t_3$ then let $C_1:=(t_3,u_0,u_1,b_1,t_1,t_3)$ and $C_2:=(u_2,a_2,b_3,u_2)$.
If $v_0 = b_3$ then let $C_1:=(b_3,u_0,u_2,b_3)$ and $C_2:=(t_1,t_2,t_3,t_1)$.
If $v_0 = a_2$ then let $C_1:=(a_2,u_0,u_2,a_2)$ and $C_2:=(t_1,b_1,u_1,b_2,t_2,t_1)$.
In all cases, we conclude that either $G$ is not \BvN, or $G$ is not \PMc, possibly both.
This completes the proof of Proposition~\ref{prop:Murty_graph_index_1}.
\end{proof}

\begin{prop}
\label{prop:W5_index_1}
Let $G$ be a simple brick, and let $e$ be a strictly thin edge of index one such that the retract of $G-e$ is the odd wheel $W_5$.
Then exactly one of the following holds:
\begin{enumerate}[(i)]
\item $G$ is the Murty graph.
\item Either $G$ is not \BvN, or $G$ is not \PMc, possibly both.
\end{enumerate}
\end{prop}
\begin{proof}
We adopt Notation~\ref{Not:strictly_thin_edge}.
The contraction vertex of $H$ resulting
from the bicontraction of $u_0$ has degree at least four, whence the hub of $H$ is indeed the contraction vertex.
We label the rim of $H$, in cyclic order, as follows: $(w_0,w_1,w_2,w_3,w_4,w_0)$. Since the hub is incident with each vertex
on the rim, we infer that one of $u_1$ and $u_2$ is incident with exactly two rim vertices, and the other is incident
with the remaining three rim vertices. Adjust notation so that $u_1$ is incident with two rim vertices; there are two cases
depending on whether these neighbours of $u_1$ appear consecutively on the rim (see Figure~\ref{fig:W5_index_1}a),
or not (see Figure~\ref{fig:W5_index_1}b).

\begin{figure}[!htb]
\centering
\begin{tikzpicture}[scale=1]

\draw (90:-0.75) -- (234:2);
\draw (90:-0.75) -- (306:2);

\draw (90:0.75) -- (18:2);
\draw (90:0.75) -- (90:2);
\draw (90:0.75) -- (162:2);

\draw (90:2) -- (162:2)node{}node[left,nodelabel]{$w_1$} -- (234:2)node{}node[left,nodelabel]{$w_2$} -- (306:2)node{}node[right,nodelabel]{$w_3$} -- (18:2)node{}node[right,nodelabel]{$w_4$} -- (90:2)node{}node[above,nodelabel]{$w_0$};

\draw (90:0.75)node{}node[above right,nodelabel]{$u_2$} -- (90:0)node{}node[right,nodelabel]{$u_0$} --
(90:-0.75)node{}node[below,nodelabel]{$u_1$};

\draw (90:-2.3)node[nodelabel]{(a)};

\end{tikzpicture}
\hspace*{0.7in}
\begin{tikzpicture}[scale=1]

\draw (90:0.75) -- (162:2);
\draw (90:0.75) -- (18:2);

\draw (90:-0.75) -- (234:2);
\draw (90:-0.75) -- (306:2);
\draw (90:-0.75) to [out=135,in=225] (90:2);

\draw (90:2) -- (162:2)node{}node[left,nodelabel]{$w_1$} -- (234:2)node{}node[left,nodelabel]{$w_2$} -- (306:2)node{}node[right,nodelabel]{$w_3$} -- (18:2)node{}node[right,nodelabel]{$w_4$} -- (90:2)node{}node[above,nodelabel]{$w_0$};

\draw (90:0.75)node{}node[above,nodelabel]{$u_1$} -- (90:0)node{}node[right,nodelabel]{$u_0$} --
(90:-0.75)node{}node[below,nodelabel]{$u_2$};

\draw (90:-2.3)node[nodelabel]{(b)};

\end{tikzpicture}
\vspace*{-0.2in}
\caption{Illustration for the proof of Proposition~\ref{prop:W5_index_1}}
\label{fig:W5_index_1}
\end{figure}

\smallskip
First suppose that the two neighbours of $u_1$ on the rim appear consecutively, and adjust notation so that $u_1w_2,u_1w_3 \in E(G)$.
Thus $u_2w_4,u_2w_0,u_2w_1 \in E(G)$. Now, up to symmetry, there are three possibilities for the end $v_0$ of the edge $e$.
These are $v_0 \in \{w_0,w_1,w_2\}$.
If $v_0=w_0$ then it is easily verified that $G$ is the Murty graph.
In the remaining two cases, we present an odd conformal bicycle $(C_1,C_2)$.
If $v_0=w_1$ then let $C_1:=(w_1,u_0,u_2,w_1)$ and let $C_2:=(w_2,u_1,w_3,w_2)$.
If $v_0=w_2$ then let $C_1:=(w_2,u_0,u_1,w_2)$ and let $C_2:=(w_0,u_2,w_1,w_0)$.

\smallskip
Now suppose that the two neighbours of $u_1$ on the rim do not appear consecutively, and adjust notation so that $u_1w_1,u_1w_4 \in E(G)$.
Thus $u_2w_0,u_2w_2,u_2w_3 \in E(G)$. Now, up to symmetry, there are three possibilities for the end $v_0$ of the edge $e$.
These are \mbox{$v_0 \in \{w_0,w_1,w_2 \}$}.
In each case, we present a conformal bicycle $(C_1,C_2)$.
If $v_0=w_0$ then let $C_1:=(w_0,u_0,u_2,w_0)$ and let $C_2:=(u_1,w_1,w_2,w_3,w_4,u_1)$.
If $v_0=w_1$ then let $C_1:=(w_1,u_0,u_1,w_1)$ and let $C_2:=(w_2,u_2,w_3,w_2)$.
If $v_0=w_2$ then let $C_1:=(w_2,u_0,u_2,w_3,w_2)$ and let $C_2:=(u_1,w_4,w_0,w_1,u_1)$.

\smallskip
This completes the proof of Proposition~\ref{prop:W5_index_1}.
\end{proof}

\begin{lem}
\label{lem:wheels_index_1}
Let $G$ be a simple brick, and let $e$ be a strictly thin edge of index one such that the retract of $G-e$ is an odd wheel $W_{2k+1}$
for some $k \geq 3$. Then $G$ has distinct edges $f_1$ and $f_2$, each of which is distinct from $e$, such that the following hold:
\begin{enumerate}[(i)]
\item $f_1$ is strictly thin of index one in $G$; let $H_1$ denote the retract of $G-f_1$;
\item $f_2$ is strictly thin of index zero in $H_1$; let $H_2:=H_1-f_2$;
\item $e$ is strictly thin of index one in $H_2$, and the retract of $H_2-e$ is the odd wheel $W_{2k-1}$.
\end{enumerate}
\end{lem}
\begin{proof}
We adopt Notation~\ref{Not:strictly_thin_edge}.
Adjust notation so that $|\partial(u_1)| \geq |\partial(u_2)|$.
The contraction vertex of $H$ resulting
from the bicontraction of $u_0$ has degree at least four, whence the hub of $H$ is indeed the contraction vertex.
We let $Q$ denote the rim of $H$. Note that $Q$ is a cycle of length $2k+1$ in $G$,
and each vertex of $Q$ is a neighbour of exactly one of $u_1$ and $u_2$.
Since $2k+1 \geq 7$, the vertex $u_1$ has at least four neighbours in $V(Q)$, consequently $|\partial(u_1)| \geq 5$.
Furthermore, since $|\partial(u_1)| \geq |\partial(u_2)|$, the cycle $Q$ has two consecutive vertices
such that each of them is a neighbour of $u_1$.

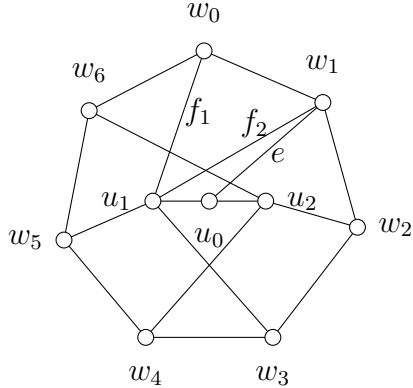
\begin{figure}[!htb]
\centering
\begin{tikzpicture}[scale=1]
\draw (34:1.1)node[nodelabel]{$e$};
\draw (97:1.2)node[nodelabel]{$f_1$};
\draw (60:1.2)node[nodelabel]{$f_2$};

\draw (0:0) -- (41:2);

\draw (180:0.75) -- (92:2);
\draw (180:0.75) -- (41:2);
\draw (0:0.75) -- (350:2);
\draw (180:0.75) -- (295:2);
\draw (0:0.75) -- (245:2);
\draw (180:0.75) -- (195:2);
\draw (0:0.75) -- (143:2);

\draw (180:0.75)node{}node[nodelabel,left]{$u_1$} -- (0:0)node{}node[nodelabel,below]{$u_0$} -- (0:0.75)node{}node[nodelabel,right]{$u_2$};

\draw (245:2) -- (195:2)node{}node[nodelabel,left]{$w_5$} -- (143:2)node{}node[nodelabel,above]{$w_6$} --
(92:2)node{}node[nodelabel,above]{$w_0$} -- (41:2)node{}node[nodelabel,above]{$w_1$} --
(350:2)node{}node[nodelabel,right]{$w_2$} -- (295:2)node{}node[nodelabel,below]{$w_3$} --
(245:2)node{}node[nodelabel,below]{$w_4$};

\end{tikzpicture}
\vspace*{-0.2in}
\caption{Illustration for the proof of Lemma~\ref{lem:wheels_index_1}}
\label{fig:wheels_index_1}
\end{figure}

\smallskip
We choose a maximal path $P$, in $Q$, such that $|V(P)| \geq 2$ and each
vertex of $P$ is a neighbour of $u_1$. Since $u_2$ has at least two neighbours in $V(Q)$, $|V(P)| \leq 2k-1$.
At least one of the ends of $P$ is distinct from the end~$v_0$ of edge $e$. We
label the vertices of $Q$, in cyclic order: $(w_0, w_1, \dots, w_{2k},w_0)$, such that $P:= (w_0, w_1, \dots, w_i)$ and $w_0 \neq v_0$.
Note that $w_1u_1,w_{2k}u_2 \in E(G)$.
We let $f_1 := w_0u_1$ and $f_2 := w_1u_1$. See Figure~\ref{fig:wheels_index_1} for an example.

\smallskip
Observe that only one end of $f_1$, namely $w_0$, is cubic.
We let $H_1$ denote the retract of $G-f_1$. Note that $H_1$ is simple, and that its unique contraction vertex is incident with the edge $f_2$.
Consequently, each end of $f_2$ is noncubic in $H_1$. We let $H_2:=H_1-f_2$. Observe that exactly one end of $e$, namely $u_0$,
is cubic in $H_2$, and that the retract of $H_2 - e$ is indeed the odd wheel $W_{2k-1}$.
By Theorem~\ref{thm:brick-expansion}, $H_2$ is a (simple) brick,
and $e$ a strictly thin edge of index one in $H_2$.
Since $H_1$ is obtained from $H_2$ by adding an edge, $H_1$ is also a (simple) brick and $f_2$ a strictly thin edge of index zero in $H_1$.
Consequently, $f_1$ is a strictly thin edge of index one in $G$.
This completes the proof of Lemma~\ref{lem:wheels_index_1}.
\end{proof}

We now have the following consequence by repeatedly applying Lemma~\ref{lem:wheels_index_1}.

\begin{cor}
\label{cor:wheels_index_1}
Let $G$ be a simple brick, and let $e$ be a strictly thin edge of index one such that the retract of $G-e$ is an odd wheel.
Then there exists a sequence $G_1, G_2, \dots, G_r$ of simple bricks such that
(i) $G_1 = W_5$;
(ii) $G_r = G$;
(iii) for $2 \leq i \leq r$, $G_i$ has a strictly thin edge $e_i$ such that $G_{i-1}$ is the retract of $G_i-e_i$; and
(iv) $e_2$ is of index one in $G_2$; and
(v) if $r \geq 3$ then $e_3$ is of index zero in $G_3$.
\qed
\end{cor}

\subsection{Index two}

\begin{prop}
\label{prop:Murty_graph_index_2}
Let $G$ be a simple brick, and let $e$ be a strictly thin edge of index two such that the retract of $G-e$ is the Murty graph.
Then $G$ is neither \BvN\ nor \PMc.
\end{prop}
\begin{proof}
We adopt Notation~\ref{Not:strictly_thin_edge}.
We label the vertices of $H$ as shown in Figure~\ref{fig:Murty_graph_index_2}. Each of the contraction vertices of  $H$,
resulting from the bicontraction of~$u_0$ and of~$v_0$, has degree at least four.
Thus $a_1$ and $a_2$ are indeed the contraction vertices.
Since $a_1a_2 \in E(H)$, we infer that $G$ has exactly one edge joining a vertex in $\{u_1,u_2\}$ with a vertex
in $\{v_1,v_2\}$.
Adjust notation so that $u_2v_2$ is an edge of $G$.

\smallskip
The vertex $v_1$ has precisely two neighbours in the set $\{b_1,b_2,b_3\}$,
and we may adjust notation so that $v_1b_2,v_1b_3 \in E(G)$.
Now, since $H$ is simple, we infer that $v_2$ is not adjacent with either of $b_2$~and~$b_3$.
Consequently, $v_2b_1 \in E(G)$.
Furthermore, $u_1$ is adjacent with at least one of $b_2$ and $b_3$, and we may adjust notation so that $u_1b_3 \in E(G)$.
Thus the graph shown in Figure~\ref{fig:Murty_graph_index_2}~(right) is a subgraph of $G$.

\begin{figure}[!htb]
\centering
\begin{tikzpicture}[scale=1]
\draw (-0.75,1.5) -- (0.75,1.5);

\draw (0.75,1.5) -- (-1.5,0) -- (-0.75,1.5) -- (0,0) -- (0.75,1.5) -- (1.5,0) -- (-0.75,1.5);

\draw (0,0) -- (0,-1);
\draw (-1.5,0) -- (-1,-2);
\draw (1.5,0) -- (1,-2);

\draw (0,-1) -- (-1,-2) -- (1,-2) -- (0,-1);

\draw (0.75,1.5)node{}node[above,nodelabel]{$a_2$};
\draw (-0.75,1.5)node{}node[above,nodelabel]{$a_1$};

\draw (0,0)node{}node[left,nodelabel]{$b_2$};
\draw (-1.5,0)node{}node[left,nodelabel]{$b_1$};
\draw (1.5,0)node{}node[right,nodelabel]{$b_3$};

\draw (0,-1)node{}node[below,nodelabel]{$t_2$};
\draw (-1,-2)node{}node[left,nodelabel]{$t_1$};
\draw (1,-2)node{}node[right,nodelabel]{$t_3$};

\draw (0,-2.5)node[nodelabel]{$H$};
\end{tikzpicture}
\hspace*{0.5in}
\begin{tikzpicture}[scale=1]
\draw (-2.3,1.5) -- (-1.7,1.9) -- (-1.1,2.3);
\draw (2.3,1.5) -- (1.7,1.9) -- (1.1,2.3);
\draw (-1.1,2.3) -- (1.1,2.3);
\draw (-1.7,1.9) -- (1.7,1.9);
\draw (2.3,1.5) -- (1.5,0);
\draw (2.3,1.5) -- (0,0);
\draw (1.1,2.3) -- (-1.5,0);
\draw (-2.3,1.5) -- (1.5,0);




\draw (0,0) -- (0,-1);
\draw (-1.5,0) -- (-1,-2);
\draw (1.5,0) -- (1,-2);

\draw (0,-1) -- (-1,-2) -- (1,-2) -- (0,-1);


\draw (0,0)node{}node[left,nodelabel]{$b_2$};
\draw (-1.5,0)node{}node[left,nodelabel]{$b_1$};
\draw (1.5,0)node{}node[right,nodelabel]{$b_3$};

\draw (0,-1)node{}node[below,nodelabel]{$t_2$};
\draw (-1,-2)node{}node[left,nodelabel]{$t_1$};
\draw (1,-2)node{}node[right,nodelabel]{$t_3$};

\draw (-2.3,1.5)node{}node[above left, nodelabel]{$u_1$};
\draw (-1.7,1.9)node{}node[above left, nodelabel]{$u_0$};
\draw (-1.1,2.3)node{}node[above, nodelabel]{$u_2$};

\draw (2.3,1.5)node{}node[above right, nodelabel]{$v_1$};
\draw (1.7,1.9)node{}node[above right, nodelabel]{$v_0$};
\draw (1.1,2.3)node{}node[above, nodelabel]{$v_2$};

\end{tikzpicture}
\vspace*{-0.2in}
\caption{Illustration for the proof of Proposition~\ref{prop:Murty_graph_index_2}}
\label{fig:Murty_graph_index_2}
\end{figure}

Now there are two possibilities:
either $u_1b_1,u_2b_2 \in E(G)$, or otherwise $u_1b_2,u_2b_1 \in E(G)$. In each case, we present an even conformal bicycle
$(C_1,C_2)$, and we leave it to the reader to find an odd conformal bicycle.
Let $C_1:=(u_2,v_2,v_0,u_0,u_2)$.
If $u_1b_1,u_2b_2 \in E(G)$ then let $C_2:=(t_1,t_2,b_2,v_1,b_3,t_3,t_1)$.
If $u_1b_2,u_2b_1 \in E(G)$ then let $C_2:=(u_1,b_2,v_1,b_3,u_1)$.
Thus $G$ is neither \BvN\ nor \PMc.
This completes the proof of Proposition~\ref{prop:Murty_graph_index_2}.
\end{proof}

\subsection{Index three}


\begin{prop}
\label{prop:W5_W7_index_3}
Let $G$ be a simple brick, and let $e$ be a strictly thin edge of index three such that the retract of $G-e$ is an odd wheel
$W_{2k+1}$, where $k \in \{2,3\}$.
If $k=2$ then $G$ is not \BvN.
If $k=3$ then either $G$ is not \BvN, or $G$ is not \PMc, possibly both.
\end{prop}
\begin{proof}
We adopt Notation~\ref{Not:strictly_thin_edge}.
The shrinking of $\{u_2,u_0,u_1,v_0,v_2\}$ results in the contraction vertex of $H$ which has degree at least five,
whence the hub of $H$ is indeed the contraction vertex.
Let $Q$ denote the rim of $H$, and label its vertices in cyclic order as follows:
$(w_0,w_1,w_2, \dots, w_{2k},w_0)$.

\smallskip
First suppose that $k=2$.
Since $|V(Q)|=5$, each of $u_2$ and $v_2$ has precisely two neighbours in $V(Q)$,
and $u_1$ has only one neighbour in $V(Q)$. Adjust notation so that $w_0u_1, w_1u_2 \in E(G)$.
See Figure~\ref{fig:W5_index_3}.
Observe that only one of the vertices in $\{w_2,w_3,w_4\}$ is a neighbour of~$u_2$,
and the remaining two are neighbours of $v_2$.
In each case, we present an odd conformal bicycle $(C_1,C_2)$.
Let $C_1 := (u_1,u_0,v_0,u_1)$.
If $w_2u_2 \in E(G)$, then let $C_2:=(w_2,u_2,w_1,w_2)$.
If $w_3u_2 \in E(G)$, then let $C_2:=(w_3,u_2,w_1,w_0,w_4,w_3)$.
If $w_4u_2 \in E(G)$, then let $C_2:=(w_2,v_2,w_3,w_2)$.
Thus, when $k=2$, $G$ is not \BvN. This proves the first part.

\begin{figure}[!htb]
\centering
\begin{tikzpicture}[scale=1]

\draw (90:2) -- (90:0.4);
\draw (162:2) -- (160:0.95);

\draw (210:0.5) -- (330:0.5);

\draw (160:0.95)node{} -- (210:0.5)node{} -- (90:0.4)node{} -- (330:0.5)node{} -- (20:0.95)node{};
\draw (160:0.95)node[nodelabel,below]{$u_2$};
\draw (20:0.95)node[nodelabel,below]{$v_2$};
\draw (210:0.5)node[nodelabel,below]{$u_0$};
\draw (330:0.5)node[nodelabel,below]{$v_0$};
\draw (50:0.5)node[nodelabel]{$u_1$};

\draw (90:2) -- (162:2)node{}node[left,nodelabel]{$w_1$} -- (234:2)node{}node[left,nodelabel]{$w_2$} -- (306:2)node{}node[right,nodelabel]{$w_3$} -- (18:2)node{}node[right,nodelabel]{$w_4$} -- (90:2)node{}node[above,nodelabel]{$w_0$};


\end{tikzpicture}
\caption{Illustration for the proof of Proposition~\ref{prop:W5_W7_index_3} when $k=2$}
\label{fig:W5_index_3}
\end{figure}

\begin{figure}[!htb]
\centering
\vspace*{-0.2in}
\begin{tikzpicture}[scale=1.2]

\draw (210:0.5) -- (330:0.5);

\draw (160:0.95) -- (143:2);
\draw (160:0.95) -- (92:2);

\draw (160:0.95)node{} -- (210:0.5)node{} -- (90:0.4)node{} -- (330:0.5)node{} -- (20:0.95)node{};
\draw (160:0.95)node[nodelabel,below]{$u_2$};
\draw (20:0.95)node[nodelabel,below]{$v_2$};
\draw (210:0.5)node[nodelabel,below]{$u_0$};
\draw (330:0.5)node[nodelabel,below]{$v_0$};
\draw (55:0.5)node[nodelabel]{$u_1$};

\draw (245:2) -- (195:2)node{}node[nodelabel,left]{$w_5$} -- (143:2)node{}node[nodelabel,above]{$w_6$} --
(92:2)node{}node[nodelabel,above]{$w_0$} -- (41:2)node{}node[nodelabel,above]{$w_1$} --
(350:2)node{}node[nodelabel,right]{$w_2$} -- (295:2)node{}node[nodelabel,below]{$w_3$} --
(245:2)node{}node[nodelabel,below]{$w_4$};

\draw (270:2.5)node[nodelabel]{(a)};
\end{tikzpicture}
\hspace*{0.2in}
\begin{tikzpicture}[scale=1.2]

\draw (210:0.5) -- (330:0.5);

\draw (160:0.95) -- (143:2);
\draw (160:0.95) -- (92:2);
\draw (20:0.95) -- (350:2);
\draw (20:0.95) to [out=280,in=20] (245:2);

\draw (160:0.95)node{} -- (210:0.5)node{} -- (90:0.4)node{} -- (330:0.5)node{} -- (20:0.95)node{};
\draw (160:0.95)node[nodelabel,below]{$u_2$};
\draw (20:0.95)node[nodelabel,above]{$v_2$};
\draw (210:0.5)node[nodelabel,below]{$u_0$};
\draw (330:0.5)node[nodelabel,below]{$v_0$};
\draw (55:0.5)node[nodelabel]{$u_1$};

\draw (245:2) -- (195:2)node{}node[nodelabel,left]{$w_5$} -- (143:2)node{}node[nodelabel,above]{$w_6$} --
(92:2)node{}node[nodelabel,above]{$w_0$} -- (41:2)node{}node[nodelabel,above]{$w_1$} --
(350:2)node{}node[nodelabel,right]{$w_2$} -- (295:2)node{}node[nodelabel,below]{$w_3$} --
(245:2)node{}node[nodelabel,below]{$w_4$};

\draw (270:2.5)node[nodelabel]{(b)};
\end{tikzpicture}
\begin{tikzpicture}[scale=1.2]

\draw (210:0.5) -- (330:0.5);

\draw (160:0.95) -- (143:2);
\draw (160:0.95) -- (41:2);
\draw (20:0.95) -- (92:2);

\draw (160:0.95)node{} -- (210:0.5)node{} -- (90:0.4)node{} -- (330:0.5)node{} -- (20:0.95)node{};
\draw (160:0.95)node[nodelabel,below]{$u_2$};
\draw (20:0.95)node[nodelabel,below]{$v_2$};
\draw (210:0.5)node[nodelabel,below]{$u_0$};
\draw (330:0.5)node[nodelabel,below]{$v_0$};
\draw (55:0.5)node[nodelabel]{$u_1$};

\draw (245:2) -- (195:2)node{}node[nodelabel,left]{$w_5$} -- (143:2)node{}node[nodelabel,above]{$w_6$} --
(92:2)node{}node[nodelabel,above]{$w_0$} -- (41:2)node{}node[nodelabel,above]{$w_1$} --
(350:2)node{}node[nodelabel,right]{$w_2$} -- (295:2)node{}node[nodelabel,below]{$w_3$} --
(245:2)node{}node[nodelabel,below]{$w_4$};

\draw (270:2.5)node[nodelabel]{(c)};
\end{tikzpicture}
\hspace*{0.2in}
\begin{tikzpicture}[scale=1.2]

\draw (210:0.5) -- (330:0.5);

\draw (160:0.95) -- (143:2);
\draw (160:0.95) -- (41:2);
\draw (90:0.4) -- (92:2);
\draw (20:0.95) -- (350:2);
\draw (90:0.4) to [out=270,in=150] (295:2);

\draw (160:0.95)node{} -- (210:0.5)node{} -- (90:0.4)node{} -- (330:0.5)node{} -- (20:0.95)node{};
\draw (160:0.95)node[nodelabel,below]{$u_2$};
\draw (20:0.95)node[nodelabel,below]{$v_2$};
\draw (210:0.5)node[nodelabel,below]{$u_0$};
\draw (330:0.5)node[nodelabel,below]{$v_0$};
\draw (55:0.5)node[nodelabel]{$u_1$};

\draw (245:2) -- (195:2)node{}node[nodelabel,left]{$w_5$} -- (143:2)node{}node[nodelabel,above]{$w_6$} --
(92:2)node{}node[nodelabel,above]{$w_0$} -- (41:2)node{}node[nodelabel,above]{$w_1$} --
(350:2)node{}node[nodelabel,right]{$w_2$} -- (295:2)node{}node[nodelabel,below]{$w_3$} --
(245:2)node{}node[nodelabel,below]{$w_4$};

\draw (270:2.5)node[nodelabel]{(d)};
\end{tikzpicture}
\begin{tikzpicture}[scale=1.2]

\draw (210:0.5) -- (330:0.5);

\draw (160:0.95) -- (143:2);
\draw (160:0.95) to [out=270,in=170] (295:2);
\draw (20:0.95) to [out=270,in=10] (245:2);

\draw (160:0.95)node{} -- (210:0.5)node{} -- (90:0.4)node{} -- (330:0.5)node{} -- (20:0.95)node{};
\draw (160:0.95)node[nodelabel,above]{$u_2$};
\draw (20:0.95)node[nodelabel,above]{$v_2$};
\draw (210:0.5)node[nodelabel,below]{$u_0$};
\draw (330:0.5)node[nodelabel,below]{$v_0$};
\draw (55:0.5)node[nodelabel]{$u_1$};

\draw (245:2) -- (195:2)node{}node[nodelabel,left]{$w_5$} -- (143:2)node{}node[nodelabel,above]{$w_6$} --
(92:2)node{}node[nodelabel,above]{$w_0$} -- (41:2)node{}node[nodelabel,above]{$w_1$} --
(350:2)node{}node[nodelabel,right]{$w_2$} -- (295:2)node{}node[nodelabel,below]{$w_3$} --
(245:2)node{}node[nodelabel,below]{$w_4$};

\draw (270:2.5)node[nodelabel]{(e)};
\end{tikzpicture}
\hspace*{0.2in}
\begin{tikzpicture}[scale=1.2]

\draw (210:0.5) -- (330:0.5);

\draw (160:0.95) -- (143:2);
\draw (160:0.95) to [out=270,in=170] (295:2);
\draw (20:0.95) to [out=270,in=10] (245:2);
\draw (90:0.4) -- (350:2);
\draw (90:0.4) -- (195:2);

\draw (160:0.95)node{} -- (210:0.5)node{} -- (90:0.4)node{} -- (330:0.5)node{} -- (20:0.95)node{};
\draw (160:0.95)node[nodelabel,above]{$u_2$};
\draw (20:0.95)node[nodelabel,above]{$v_2$};
\draw (210:0.5)node[nodelabel,below]{$u_0$};
\draw (330:0.5)node[nodelabel,below]{$v_0$};
\draw (55:0.5)node[nodelabel]{$u_1$};

\draw (245:2) -- (195:2)node{}node[nodelabel,left]{$w_5$} -- (143:2)node{}node[nodelabel,above]{$w_6$} --
(92:2)node{}node[nodelabel,above]{$w_0$} -- (41:2)node{}node[nodelabel,above]{$w_1$} --
(350:2)node{}node[nodelabel,right]{$w_2$} -- (295:2)node{}node[nodelabel,below]{$w_3$} --
(245:2)node{}node[nodelabel,below]{$w_4$};

\draw (270:2.5)node[nodelabel]{(f)};
\end{tikzpicture}
\vspace*{-0.2in}
\caption{Illustration for the proof of Proposition~\ref{prop:W5_W7_index_3} when $k=3$}
\label{fig:W7_index_3}
\end{figure}
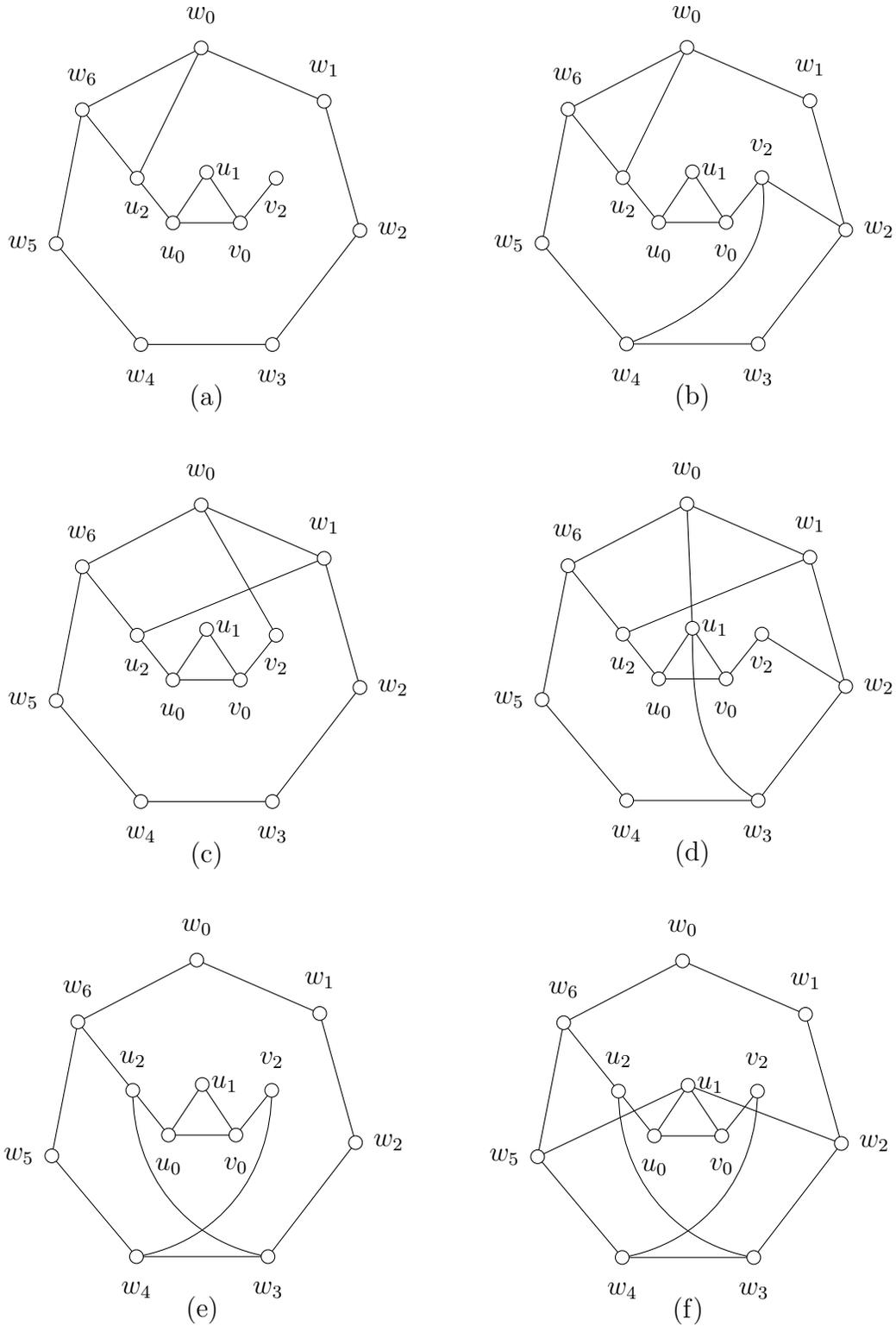

\smallskip
Now suppose that $k=3$.\footnote{Unlike the $k=2$ case, there are too many possibilities for the graph $G$. More precisely, there are
$46$ nonisomorphic possibilities as can be checked by computations.}
First, let us assume that either $u_2$, or $v_2$, has two neighbours that appear consecutively on the rim $Q$. Adjust notation so that
$u_2w_6,u_2w_0 \in E(G)$. See Figure~\ref{fig:W7_index_3}(a).
Observe that, if $v_2$ is a neighbour of any vertex in $\{w_1, w_3, w_5\}$ then the two triangles $(w_0,u_2,w_6,w_0)$
and $(u_1,u_0,v_0,u_1)$ constitute an odd conformal bicycle, and we are done.
So we may assume that $v_2$ is not a neighbour of any vertex in $\{w_1,w_3,w_5\}$, whence $v_2w_2,v_2w_4 \in E(G)$.
See Figure~\ref{fig:W7_index_3}(b).
Now, $w_1$~is either a neighbour of $u_2$ or of $u_1$.
If $w_1u_2 \in E(G)$, then the two triangles $(w_1,u_2,w_0,w_1)$ and $(u_1,u_0,v_0,u_1)$ constitute an odd conformal bicycle.
If $w_1u_1 \in E(G)$, then the $6$-cycle $(w_1,u_1,v_0,u_0,u_2,w_0,w_1)$ and the $4$-cycle $(w_2,v_2,w_4,w_3,w_2)$
constitute an even conformal bicycle. Thus, in either case, we are done.

\smallskip
Next, let us assume that $u_2$ has two neighbours $w_i$ and $w_{i+2}$ on the rim
and that $v_2$ is a neighbour of $w_{i+1}$, where subscript arithmetic is done modulo $7$.
See Figure~\ref{fig:W7_index_3}(c) for an example.
Then the triangle $(u_1,u_0,v_0,u_1)$ and the $7$-cycle $Q+w_iu_2+w_{i+2}u_2-w_iw_{i+1}-w_{i+1}w_{i+2}$
constitute an odd conformal bicycle, and we are done.

\smallskip
Henceforth, we may assume the following. Neither $u_2$ nor $v_2$ has two neighbours that appear consecutively on the rim $Q$.
Furthemore, if $u_2w_i, u_2w_{i+2} \in E(G)$ then $u_1w_{i+1} \in E(G)$.
Likewise, if $v_2w_i,v_2w_{i+2} \in E(G)$ then $u_1w_{i+1} \in E(G)$.

\smallskip
Let us assume that $u_2$ has two neighbours $w_i$ and $w_{i+2}$ on the rim.
Adjust notation so that $u_2w_6,u_2w_1 \in E(G)$.
It follows from the aforementioned assumptions that $u_1w_0 \in E(G)$.
See Figure~\ref{fig:W7_index_3}(d).
Furthermore, at most one of $w_3$ and $w_4$ is a neighbour of $v_2$.
Since $v_2$ has at least two neighbours in $V(Q)$, at least one of $w_2$ and $w_5$
is a neighbour of $v_2$.
By symmetry, we may adjust notation so that $v_2w_2 \in E(G)$.
Consequently, $w_3$ is not a neighbour of $v_2$.
Furthermore, since $w_1u_2, w_2v_2 \in E(G)$, $w_3$ is not a neighbour of $u_2$.
Thus $w_3u_1 \in E(G)$.
Observe that the $6$-cycle $(w_3,u_1,u_0,v_0,v_2,w_2,w_3)$ and the $4$-cycle $(w_0,w_6,u_2,w_1,w_0)$
constitute an even conformal bicycle, and we are done.

\smallskip
Henceforth we may assume the following. If either $u_2$, or $v_2$,
has two neighbours $w_i,w_j \in V(Q)$ then the distance between $w_i$ and $w_j$ on the rim is exactly three.
Consequently, $u_2$~and~$v_2$ are cubic vertices of $G$.

\smallskip
Adjust notation so that $u_2w_3,u_2w_6 \in E(G)$.
See Figure~\ref{fig:W7_index_3}(e).
It follows from our assumption that at least one of $w_4$ and $w_5$ is a neighbour of $v_2$.
By symmetry, we may adjust notation so that $w_4v_2 \in E(G)$.
Consequently, $w_2$ and $w_5$ are both neighbours of $u_1$.
See Figure~\ref{fig:W7_index_3}(f).
Now the two $6$-cycles $(w_2,u_1,w_5,w_6,w_0,w_1,w_2)$ and $C_2:=(w_3,u_2,u_0,v_0,v_2,w_4,w_3)$
constitute an even conformal bicycle.

\smallskip
Thus, in every case, we have found a conformal bicycle in $G$, whence either $G$ is not \BvN, or $G$ is not \PMc, possibly both.
This proves the second part of Proposition~\ref{prop:W5_W7_index_3}.
\end{proof}

\begin{lem}
\label{lem:wheels_index_3}
Let $G$ be a simple brick, and let $e$ be a strictly thin edge of index three such that the retract of $G-e$ is an odd wheel $W_{2k+1}$
for some $k \geq 4$. Then $G$ has distinct edges $f_1$ and $f_2$, each of which is distinct from $e$, such that the following hold:
\begin{enumerate}[(i)]
\item $f_1$ is strictly thin of index one in $G$; let $H_1$ denote the retract of $G-f_1$;
\item $f_2$ is strictly thin of index zero in $H_1$; let $H_2:=H_1-f_2$;
\item $e$ is strictly thin of index three in $H_2$, and the retract of $H_2-e$ is the odd wheel $W_{2k-1}$.
\end{enumerate}
\end{lem}
\begin{proof}
We adopt Notation~\ref{Not:strictly_thin_edge}.
The shrinking of $\{u_2,u_0,u_1,v_0,v_2\}$ results in the contraction vertex of $H$ which has degree at least five,
whence the hub of $H$ is indeed the contraction vertex.
We let $Q$ denote the rim of $H$,
and we label its vertices in cyclic order as follows: $(w_0, w_1, \dots, w_{2k},w_0)$.
Thus $Q$ is a cycle of length $2k+1$ in $G$.
We let $S:=\{u_2,u_1,v_2\}$.
Since $k \geq 4$, the sum of the degrees of members of~$S$ is at least~$13$;
this proves the following.
\begin{sta}
\label{sta:degrees-of-S}
At most two members of~$S$ are cubic, and at least one member of~$S$
has degree five or more. \qed
\end{sta}

Each vertex in $V(Q)$ is a neighbour of exactly one vertex in $S$.
This helps us define a function
$\sigma:V(Q) \rightarrow S$ as follows.
For each $w_i \in V(Q)$, let $\sigma(w_i)$ denote the unique
vertex in $S$ that is a neighbour of $w_i$.
Now we shall establish the following; it will help us in locating
a pair of edges $f_1$ and $f_2$ that satisfy statements {\it (i), (ii)} and {\it (iii)}.
\begin{sta}
\label{sta:existence-of-w}
There exists $w_i \in V(Q)$ such that:
\begin{enumerate}[(a)]
\item $\sigma(w_i)$ has degree four or more in~$G$,
\item $\sigma(w_{i-1})
\neq\footnote{All of the subscript arithmetic is done modulo $2k+1$.}
~\sigma(w_{i+1})$, and
\item at least one of $\sigma(w_{i-1})$ and $\sigma(w_{i+1})$
has degree four or more in~$G-\sigma(w_i)w_i$,
\end{enumerate}
\end{sta}
\begin{proof}
By \ref{sta:degrees-of-S}, the set~$S$ has zero, one or two cubic vertices.

\smallskip
First we consider the case: $S$ has zero cubic vertices.
Since $Q$ is an odd cycle, and since each member of~$S$ has at least
one neighbour in~$V(Q)$, it follows that there
exists $w_i \in V(Q)$
such that $\sigma(w_{i-1}) \neq \sigma(w_{i+1})$;
whence at least one of $\sigma(w_{i-1})$ and $\sigma(w_{i+1})$
is different from~$\sigma(w_i)$; this yields the desired conclusion.

\smallskip
Now we consider the case: $S$ has precisely one cubic vertex, say~$z$.
The graph~$Q-N(z)$ has at most two components; each of these components
is a path; at least one of them, say~$P$, has four or more vertices.
Adjust notation so that $P:=(w_i, w_{i+1}, \dots, w_{j-1},w_j)$ where $i < j$.
Note that $w_i$ and $w_j$ both satisfy statements {\it (a)} and {\it (b)}.
Furthermore, if either $\sigma(w_i)$ has degree five or more, or if $\sigma(w_i) \neq \sigma(w_{i+1})$, then $w_i$ satisfies statement {\it (c)}.
Likewise, if $\sigma(w_j)$ has degree five or more,
or if $\sigma(w_j) \neq \sigma(w_{j-1})$,
then $w_j$ satisfies statement {\it (c)}.
Now suppose that none of these conditions hold;
in other words, $\sigma(w_i)=\sigma(w_{i+1})$ and $\sigma(w_j)=\sigma(w_{j-1})$
are vertices of degree precisely four.
Observe that $\sigma(w_i) \neq \sigma(w_j)$ --- since otherwise the degree of $\sigma(w_i)$ is five or more.
Consequently,
$S=\{z,\sigma(w_i),\sigma(w_j)\}$; this contradicts~\ref{sta:degrees-of-S}.

\smallskip
Finally, we consider the case: $S$ has precisely two cubic vertices, say~$y$ and $z$.
The graph~$Q-N(y)-N(z)$ has at most four components; each of these components is
a path; at least one of them, say~$P$, has two or more vertices.
We let $w_i$ denote an end of~$P$.
By \ref{sta:degrees-of-S}, the vertex~$\sigma(w_i)$ has degree at least five;
note that $w_i$ satisfies statements {\it (a), (b)} and {\it (c)}.

\smallskip
This completes the proof of \ref{sta:existence-of-w}.
\end{proof}

We invoke \ref{sta:existence-of-w}, and adjust notation so that $\sigma(w_1)$
has degree four or more (in~$G$), $\sigma(w_0) \neq \sigma(w_2)$
and $\sigma(w_2)$ has degree four or more in~$G-\sigma(w_1)w_1$.
We let $f_1:=\sigma(w_1)w_1$ and $f_2:=\sigma(w_2)w_2$.

\smallskip
Observe that only one end of $f_1$, namely $w_1$, is cubic. We let $H_1$ denote the retract of $G-f_1$.
Note that $H_1$ is simple since $\sigma(w_0) \neq \sigma(w_{2})$, and that its unique contraction vertex is incident
with the edge $f_2$.
Also, since $\sigma(w_2)$ has degree at least four in $G-f_1$,
each end of $f_2$ is noncubic in $H_1$.
We let $H_2 := H_1- f_2$.
Observe that both ends of $e$ are cubic in~$H_2$, and that the retract of $H_2-e$ is
indeed the odd wheel $W_{2k-1}$.
By Theorem~\ref{thm:brick-expansion},
$H_2$ is a (simple) brick, and $e$ is a strictly thin edge of
index three in $H_2$.
Since $H_1$ is obtained from $H_2$ by adding an edge,
$H_1$ is also a (simple) brick
and $f_2$ a strictly thin edge of index zero in $H_1$.
Consequently, $f_1$ is a strictly thin edge of index one in $G$.
This completes the proof of Lemma~\ref{lem:wheels_index_3}.
\end{proof}

We now have the following consequence by repeatedly applying Lemma~\ref{lem:wheels_index_3}.

\begin{cor}
\label{cor:wheels_index_3}
Let $G$ be a simple brick, and let $e$ be a strictly thin edge of index three such that the retract of $G-e$ is an odd wheel $W_{2k+1}$
where $k \geq 3$.
Then there exists a sequence $G_1, G_2, \dots, G_r$ of simple bricks such that
(i) $G_1$ is $W_7$;
(ii) $G_r = G$;
(iii) for $2 \leq i \leq r$, $G_i$ has a strictly thin edge $e_i$ such that $G_{i-1}$ is the retract of $G_i-e_i$; and
(iv) $e_2$ is of index three in~$G_2$.
\qed
\end{cor}

\subsection{Proof of Theorem~\ref{thm:BvN-PMc-bricks}}

\begin{proofOf}{Theorem~\ref{thm:BvN-PMc-bricks}}
Let $G$ be a brick that is \BvN\ as well as \PMc. We induct on the number of edges.

\smallskip
First suppose that either $G$ is not simple and let $e$ denote a multiple edge of $G$,
or otherwise $G$ is a simple brick that has a strictly thin edge $e$ of index zero.
In either case, \mbox{$G-e$} is a smaller brick that is also \BvN\ and \PMc. By the induction hypothesis,
$G-e$ is one of the graphs listed in Theorem~\ref{thm:BvN-PMc-bricks}. The result follows from
Propositions~\ref{prop:wheel_plus_edge}~and~\ref{prop:Murty_graph_plus_edge}.

\smallskip
Now suppose that $G$ is a simple brick that is free of strictly thin edges of index zero.
If $G$ is a \NT\ brick then we are done by Proposition~\ref{prop:NT-brick-BvN-PMc}.
Now suppose that $G$ is not a \NT\ brick.
In this case, we invoke Theorem~\ref{thm:strictly-thin-edge};
whence $G$ has a strictly thin edge $e$ of index at least one.
We let $H$ denote the retract of $G-e$. Thus $H$ is a simple brick
that is \BvN\ as well as \PMc. By the induction hypothesis, $H$ is either an odd wheel, or it is the Murty graph.
We consider three cases depending on the index of $e$.

\medskip
\noindent
{\bf Case 1:} Edge $e$ is of index one.

\smallskip
\noindent
It follows from Proposition~\ref{prop:brick-minus-thin-edge-cases} that $H$ has just one contraction vertex, and its degree is
at least four. By the induction hypothesis, $H$ is either an odd wheel of order at least six, or $H$ is the Murty graph.
In the latter case, we arrive at a contradition by invoking Proposition~\ref{prop:Murty_graph_index_1}.

\smallskip
Now suppose that $H$ is an odd wheel. Corollary~\ref{cor:wheels_index_1} implies that there exists a sequence
$G_1, G_2, \dots, G_r$ of simple bricks such that (i) $G_1 = W_5$;
(ii) $G_r = G$;
(iii) for $2 \leq i \leq r$, $G_i$~has a strictly thin edge $e_i$ such that $G_{i-1}$ is the retract of $G_i-e_i$;
(iv) $e_2$ is of index one in~$G_2$; and
(v) if $r \geq 3$ then $e_3$ is of index zero in $G_3$.

\smallskip
Note that all of the bricks $G_1, G_2, \dots, G_r$ are \BvN\ and \PMc.
Since $G_1$ is $W_5$, Proposition~\ref{prop:W5_index_1} implies that $G_2$ is indeed the Murty graph.
If $r=2$ then we are done.
Now suppose that $r \geq 3$. Thus $G_3$ is obtained from $G_2$ by adding an edge joining two nonadjacent vertices.
We invoke Proposition~\ref{prop:Murty_graph_plus_edge} to arrive at a contradiction.

\medskip
\noindent
{\bf Case 2:} Edge $e$ is of index two.

\smallskip
\noindent
It follows from Proposition~\ref{prop:brick-minus-thin-edge-cases} that $H$ has two contraction vertices, each of which
has degree at least four. By the induction hypothesis, $H$ is in fact the Murty graph.
We invoke Proposition~\ref{prop:Murty_graph_index_2} to arrive at a contradiction.

\medskip
\noindent
{\bf Case 3:} Edge $e$ is of index three.

\smallskip
\noindent
It follows from Proposition~\ref{prop:brick-minus-thin-edge-cases} that $H$ has just one contraction vertex, and its degree is
at least five. By the induction hypothesis, $H$ is in fact an odd wheel of order at least six.
If $H=W_5$, we arrive at a contradiction by invoking Proposition~\ref{prop:W5_W7_index_3}.

\smallskip
Now suppose that $H=W_{2k+1}$ where $k \geq 3$.
Corollary~\ref{cor:wheels_index_3} implies that there exists a sequence $G_1, G_2, \dots, G_r$ of simple bricks such that
(i) $G_1=W_7$;
(ii) $G_r = G$;
(iii) for $2 \leq i \leq r$, $G_i$ has a strictly thin edge $e_i$ such that $G_{i-1}$ is the retract of $G_i-e_i$; and
(iv) $e_2$ is of index three in~$G_2$.

\smallskip
Note that all of the bricks $G_1, G_2, \dots, G_r$ are \BvN\ and \PMc.
However, since $G_1$ is $W_7$, Proposition~\ref{prop:W5_W7_index_3} implies that either $G_2$ is not \BvN, or $G_2$ is not \PMc, possibly both.
We thus have a contradiction.

\medskip
This completes the proof of Theorem~\ref{thm:BvN-PMc-bricks}.
\end{proofOf}

\subsection{An infinite family}

Thus far we have given a complete characterization of matching covered graphs that are \BvN\ as well as \PMc.
In a recent paper, Lucchesi, Carvalho, Kothari and Murty \cite{lckm18} showed that the problem of characterizing
\BvN\ bricks\footnote{As mentioned earlier, this class is the same as the class of solid bricks.}
is equivalent to the problem of characterizing $\overline{C_6}$-free bricks.
In the same work, they present two infinite families of \BvN\ bricks; see \cite[Figures 7 and 8]{lckm18}.
The members of these families, except for the Murty graph, are not \PMc\ (as expected).

\smallskip
Now we present a family of \PMc\ bricks that are not \BvN.
For $k \geq 2$, let $W_{2k+1}$ denote the odd wheel and let $Q$ denote its rim.
We label the vertices of $Q$, in cyclic order, as follows: $(w_0, w_1, \dots, w_{2k})$.
We obtain a new graph from $W_{2k+1}$ as follows.

\smallskip
First split the hub of $W_{2k+1}$ into three pairwise nonadjacent vertices $u_2,u_1$ and $v_2$ and distribute
the spokes of the wheel so that:
\begin{itemize}
\item $v_2$ is adjacent to $w_0$ and $w_4$;
\item $u_1$ is adjacent to $w_2$; and
\item $u_2$ is adjacent to the remaining $2k-2$ vertices; that is, to $w_1$, to $w_3$, and if $k \geq 3$ then to each of
$w_5, w_6, \dots, w_{2k}$.
\end{itemize}

Now add two new vertices $u_0$ and $v_0$, and the following edges: $u_0u_2,u_0u_1,u_0v_0,v_0u_1,v_0v_2$.
The graph obtained in this manner is denoted by $\mathcal{P}_{2k+1}$.
In Figure~\ref{fig:PMc_but_not_BvN}, we show the brick $\mathcal{P}_7$.
Observe that the retract of $\mathcal{P}_{2k+1} - u_0v_0$ is the odd wheel $W_{2k+1}$.
It follows from Theorem~\ref{thm:brick-expansion} that $\mathcal{P}_{2k+1}$ is indeed a brick.

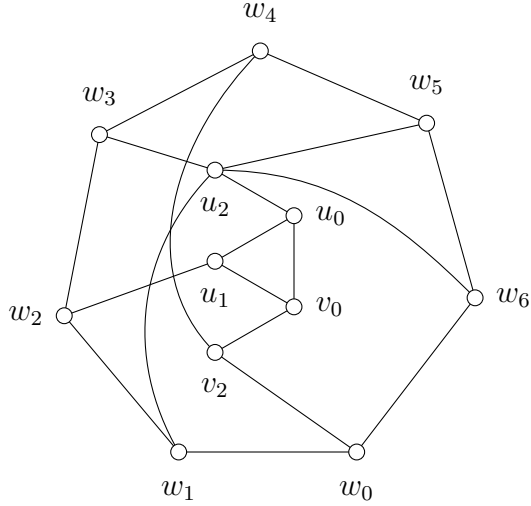
\begin{figure}[!htb]
\centering
\begin{tikzpicture}[scale=0.7]

\draw (120:2) -- (143:4);
\draw (120:2) -- (41:4);
\draw (120:2) to [out=0,in=135] (350:4);
\draw (120:2) to [out=225,in=120] (245:4);

\draw (0:-1) -- (195:4);

\draw (240:2) -- (295:4);
\draw (240:2) to [out=135,in=225] (92:4);

\draw (60:1) -- (120:2)node{}node[nodelabel,below]{$u_2$};
\draw (-60:1) -- (240:2)node{}node[nodelabel,below]{$v_2$};
\draw (0:-1) -- (60:1)node{}node[nodelabel,right]{$u_0$} -- (-60:1)node{}node[nodelabel,right]{$v_0$} -- (0:-1)node{}node[nodelabel,below]{$u_1$};

\draw (245:4) -- (195:4)node{}node[nodelabel,left]{$w_2$} -- (143:4)node{}node[nodelabel,above]{$w_3$} -- (92:4)node{}node[nodelabel,above]{$w_4$} -- (41:4)node{}node[nodelabel,above]{$w_5$} -- (350:4)node{}node[nodelabel,right]{$w_6$} -- (295:4)node{}node[nodelabel,below]{$w_0$} -- (245:4)node{}node[nodelabel,below]{$w_1$};

\end{tikzpicture}
\vspace*{-0.2in}
\caption{The brick $\mathcal{P}_7$}
\label{fig:PMc_but_not_BvN}
\end{figure}

\begin{prop}
\label{prop:PMc_but_not_BvN}
For $k \geq 2$, the brick $\mathcal{P}_{2k+1}$ is \PMc, but it is not \BvN.
\end{prop}
\begin{proof}
Let $C_1:=(v_0,u_1,u_0,v_0)$.
If $k=2$ then let $C_2:=(w_0,v_2,w_4,w_0)$; otherwise let $C_2:=(w_0,v_2,w_4,w_5,w_6,\dots,w_{2k},w_0)$.
Observe that $(C_1,C_2)$ is an odd conformal bicycle. Thus $\mathcal{P}_{2k+1}$ is not \BvN.

\smallskip
Now let $G:=\mathcal{P}_{2k+1}$ for some $k \geq 2$, and suppose (for the sake of contradiction) that $G$ is not \PMc.
Consequently, $G$ has an even conformal bicycle $(C_1,C_2)$. Observe that $G-v_2-u_2$ is free of even cycles.
Thus one of $C_1$ and $C_2$ contains $v_2$ whereas the other contains $u_2$.
Adjust notation so that $v_2 \in V(C_1)$ and $u_2 \in V(C_2)$.

\smallskip
We let $D:=\partial(\{v_0,u_1,u_0\})$. Note that $D$ is a $3$-cut.
Since $G$ is simple, we may also view a cycle as a set of edges without any ambiguity.
We will use the fact that $|C_1 \cap D|$ and $|C_2 \cap D|$
are both even. Since $u_2 \in V(C_2)$, we infer that either $C_1 \cap D = \{v_2v_0,u_1w_2\}$ or otherwise $C_1 \cap D = \emptyset$.

\smallskip
First consider the case in which $C_1 \cap D = \emptyset$. Thus $v_2w_0,v_2w_4 \in C_1$. Since $C_1$ is an even cycle, $C_1$ is now
uniquely determined; more specifically, $C_1 := (w_0,v_2,w_4,w_3,w_2,w_1,w_0)$.
Now we observe that $C_2 \cap D=\emptyset$.
However, $u_2 \in V(C_2)$.
Consequently, the triangle $(v_0,u_1,u_0,v_0)$ is a component of $G-V(C_1)-V(C_2)$. Thus the pair $(C_1,C_2)$ is not conformal,
contrary to our assumption.

\smallskip
Now consider the case in which $C_1 \cap D = \{v_2v_0,u_1w_2\}$. It follows that \mbox{$C_2 \cap D = \emptyset$}.
If $u_0 \notin V(C_1)$ then $u_0$ is an isolated vertex in $G-V(C_1)-V(C_2)$, contrary to our assumption.
Thus $u_0 \in V(C_1)$, whence $v_0u_0,u_1u_0 \in C_1$. Note that $C_1$ contains exactly one of $v_2w_0$ and $v_2w_4$.
However, these two cases are symmetric. Adjust notation so that $v_2w_0 \in C_1$. Since $C_1$ is an even cycle,
$C_1$ is again determined uniquely; more specifically, $C_1 := (w_0,v_2,v_0,u_0,u_1,w_2,w_3,w_4, \dots, w_{2k},w_0)$.
Now observe that $G-V(C_1)$ has precisely two vertices --- namely, $u_2$ and $w_1$. This contradicts the
existence of $C_2$.

\smallskip
Since each case leads us to a contradiction, we conclude that $\mathcal{P}_{2k+1}$ is indeed \PMc.
This completes the proof of Proposition~\ref{prop:PMc_but_not_BvN}.
\end{proof}

\smallskip
\noindent
{\bf Acknowledgements:} We would like to extend our gratitude to one of the anonymous referees who provided
detailed feedback and also observed a bug in the ``proof'' of Lemma~3.18 in the original manuscript.

\bibliographystyle{plain}
\bibliography{clm}

\end{document}